\documentclass[a4paper, 11pt, english, reqno, dvipsnames]{amsart} 

\usepackage{amsmath}
\usepackage{amsthm}
\usepackage{amssymb}

\usepackage{mathtools}  
\usepackage{mathrsfs}   
\usepackage{amsfonts}
\usepackage{mathdots}  
\usepackage[bb=libus]{mathalpha} 
\usepackage{amsbsy}

\usepackage{tikz}
\usepackage{tikz-cd}

\usepackage[all]{xy}
\usepackage{comment}

\usepackage{quiver}

\usepackage{lipsum}

\usepackage[utf8]{inputenc}
\usepackage[T1]{fontenc}

\theoremstyle{plain}
\newtheorem{theorem}{Theorem}[section]
\newtheorem{proposition}[theorem]{Proposition}
\newtheorem{lemma}[theorem]{Lemma}

\newtheorem{definition-theorem}[theorem]{Definition-Theorem}
\newtheorem{definition-proposition}[theorem]{Definition-Proposition}
\newtheorem{maintheorem}{Main Theorem}

\theoremstyle{definition}

\newtheorem{definition}[theorem]{Definition}

\newtheorem{example}[theorem]{Example}

\theoremstyle{remark}
\newtheorem{remark}[theorem]{Remark}

\newcounter{manualconj}
\makeatletter
\newenvironment{myconjecture}[1]
  {%
    \refstepcounter{manualconj}%
    \def\@currentlabel{#1}%
    \trivlist
    \item[\hskip \labelsep {\bfseries Conjecture #1.}\hskip \labelsep]%
    \itshape
  }
  {\endtrivlist}
\makeatother

\usepackage{enumitem} 

\usepackage[figurewithin=none]{caption}
\usepackage{graphicx}
\usepackage{float}
\usepackage{xstring}
\usepackage{array}
\usepackage{longtable}

\makeatletter
\let\origitem\item

\newcommand{\autodescitem}[1][]{%
  \def\@currentlabel{#1}%
  \origitem[#1]%
}
\makeatother
\newenvironment{condition}{%
  \description
  \let\item\autodescitem
 
}{%
  \enddescription
}
\newcommand{\dref}[1]{\textup{\ref{#1}}}

\usepackage{xcolor}

\usepackage{mycom}
\usepackage{myarrows}

\makeatletter
\def\l@subsection{\@tocline{2}{0pt}{2pc}{5pc}{}}
\makeatother

\usepackage[style=alphabetic,sorting=nyt, backend=biber, backref=true]{biblatex}
\usepackage{url}

\usepackage[unicode, colorlinks, linktocpage=true]{hyperref}
\hypersetup{
	linkcolor = blue,
	citecolor = Red,
	urlcolor = PineGreen,
	bookmarksnumbered = true
}

\DeclareMathOperator{\Ass}{Ass}
\DeclareMathOperator{\Ann}{Ann}
\DeclareMathOperator{\Ht}{ht}

\usepackage{orcidlink}
\title[Pseudonullity of fine Selmer groups]{Pseudonullity for fine Selmer groups over multiple $\ZZ_p^{d}$-extensions}
\author{Peikai Qi \,\orcidlink{0009-0004-4319-4359}, Ruichen Xu \,\orcidlink{0009-0003-4555-2116}}
\address[Peikai Qi]{Morningside Center of Mathematics, Chinese Academy of Sciences, No. 55 Zhongguancun East Road, Beijing, 100190, China.}
\email{qipeikai@amss.ac.cn}
\address[Ruichen Xu]{Academy of Mathematics and Systems Science, Chinese Academy of Sciences, No. 55 Zhongguancun East Road, Beijing, 100190, China.}
\email{xuruichen@amss.ac.cn}
\date{\today}

\begin{document}
\begin{abstract}
In this article, we develop vertical control theorems for fine Selmer groups associated with Galois representations and their deformations in $p$-adic families. We apply them to provide new evidence for \emph{Conjectures A and B} of Coates--Sujatha and to obtain pseudonullity of the fine Selmer groups for cyclotomic characters, elliptic curves, classical cuspidal newforms, and Hida families of $p$-ordinary cuspidal newforms.
\end{abstract}

    \maketitle

\begingroup
\renewcommand{\thefootnote}{}
\footnotetext{Peikai Qi and Ruichen Xu contributed equally to this article.}
\endgroup

    \tableofcontents

\section{Introduction}

Let $F$ be a number field and let $p$ be an odd prime. We denote by $F_{\cyc}$ the cyclotomic $\ZZ_p$-extension of $F$. Let $K(F_{\cyc})$ be the maximal unramified pro-$p$ extension of $F_{\cyc}$ in which every prime above $p$ splits completely. A fundamental result of K.~Iwasawa asserts that $\Gal(K(F_{\cyc})/F_{\cyc})$ is a finitely generated torsion module over the Iwasawa algebra $\ZZ_p\lrbracket{\Gal(F_{\cyc}/F)}$. Iwasawa further conjectured that this Galois group is finitely generated as a $\ZZ_p$-module; see \cite{iwasawa1973mu, iwasawa1973extensions}.

In parallel with Iwasawa's conjecture, J.~Coates and R.~Sujatha \cite{coates2005fine} formulated their celebrated \emph{Conjecture A} for elliptic curves, which predicts the finiteness (as $\ZZ_p$-modules) of fine Selmer groups over the cyclotomic $\ZZ_p$-extension $F_{\cyc}$; see below. They further proposed \emph{Conjecture B}, asserting the pseudonullity of fine Selmer groups over admissible $p$-adic Lie extensions containing $F_{\cyc}$; see below. These conjectures have subsequently been generalized to wider classes of Galois representations by various authors, including Sujatha, S.~Jha, and M.~F.~Lim; see \cite{jha2011joa, jha2012asian, lim2017}.

The present work is motivated by the philosophy that arithmetic properties of fine Selmer groups should propagate vertically in towers of $p$-adic Lie extensions. Namely, starting from a cyclotomic-like $\ZZ_p$-extension $F_{\infty}/F$ contained in a $\ZZ_p^d$-extension $\tilF/F$, we investigate whether the \emph{finiteness property} (corresponding to \emph{Conjecture A}) and \emph{pseudonullity properties} (corresponding to \emph{Conjecture B}) of fine Selmer groups over $F_{\infty}$ propagate to the corresponding fine Selmer groups over $\tilF$.

\subsection{Conjectures for fine Selmer groups}
We introduce some conjectures for fine Selmer groups. Let $F$ be a number field and let $S$ be a finite set of places of $F$ containing $p$.

In most sections of the paper, we assume only that $R$ is a commutative Noetherian profinite $\ZZ_p$-algebra  (e.g. $R = \ZZ_p$, $\ZZ_p\lrbracket{\Gal(F_{\cyc}/F)}$, deformation rings, etc). For simplicity in stating the main results in the introduction, we assume that $R$ is a commutative complete Noetherian local ring with residual characteristic $p$. Let $T$ be a finitely generated $R$-module, equipped with a continuous $R$-linear $\Gal_{F}$-action unramified outside $S$. We write $A := T \otimes_{R} \Hom_{\ZZ_p}(R, \QQ_p/\ZZ_p)$. For any algebraic (possibly infinite) extension $\calL$ of $F$ unramified outside $S$, the \emph{fine Selmer group} of $A$ over $\calL$ with respect to $S$ is defined to be
\[
\Sel_{S}(A/\calL) := \ker\left(\rmH^{1}(\Gal_{\calL, S}, A) \rightarrow \bigoplus_{v \in S} \left( \varinjlim_{L} \bigoplus_{w \mid v} \rmH^{1}(L_w, A) \right) \right),
\]
where $L$ runs through all finite extensions of $F$ contained in $\calL$ and $w$ runs through all places of $L$ dividing $v$. We will write 
\[
Y_{S}(A/\calL) := \Hom_{\cts}(\Sel_{S}(A/\calL), \QQ_p/\ZZ_p)
\]
for the Pontryagin dual $\Sel_{S}(A/\calL)^{\vee}$ of the Selmer group. We first state the A-type conjecture.
\begin{myconjecture}{A} \label{myconj:A}
  For any number field $F$, the dual fine Selmer group $Y_{S}(A/F_{\cyc})$ is a finitely generated $R$-module.
\end{myconjecture}

We generalize this conjecture to arbitrary admissible $p$-adic Lie extensions containing $F_{\cyc}/F$, which we call the B-type conjectures. 

\begin{myconjecture}{B${}^{-}$} \label{myconj:Bminus}
   Let $F$ be any number field, and let $\tilF/F$ be any admissible $p$-adic Lie extension of dimension $> 1$ containing $F_{\cyc}/F$. Assuming Conjecture \ref{myconj:A}, the module $Y_{S}(A/\tilF)$ is finitely generated over $R\lrbracket{\Gal(\tilF/F_{\cyc})}$.
\end{myconjecture}

We regard it as weaker than Conjecture \ref{myconj:B}, stated below, which was originally formulated in \cite[Conjecture B]{coates2005fine} for elliptic curves.

\begin{myconjecture}{B} \label{myconj:B}
    Let $F$ be any number field, and let $\tilF/F$ be any admissible $p$-adic Lie extension of dimension $> 1$ containing $F_{\cyc}/F$. Assuming Conjecture \ref{myconj:A}, the dual fine Selmer group $Y_{S}(A/\tilF)$ is a pseudonull $R\lrbracket{\Gal(\tilF/F)}$-module.
\end{myconjecture}

In this article, we formulate a \emph{stronger form of Conjecture \ref{myconj:B}} as follows.
\begin{myconjecture}{B${}^{+}$} \label{myconj:Bplus}
    Let $F$ be any number field, and let $\tilF/F$ be any admissible $p$-adic Lie extension of dimension $> 1$ containing $F_{\cyc}/F$. Assuming that $Y_{S}(A/F_{\cyc})$ is a pseudonull $R\lrbracket{\Gal(F_{\cyc}/F)}$-module, the dual fine Selmer group $Y_{S}(A/\tilF)$ is a pseudonull $R\lrbracket{\Gal(\tilF/F)}$-module.
\end{myconjecture}

In this article, our goal is to provide evidence for Conjectures \ref{myconj:Bminus} and \ref{myconj:Bplus}. Our original motivation for studying Conjecture \ref{myconj:Bplus} dates back to the work of S.~Fujii \cite{fujii2017on}. Let $F / F^{+}$ be a CM field extension and let $\tilF/F$ denote the compositum of all $\ZZ_p$-extensions of $F$. Fujii constructs a special $\ZZ_p$-extension of $F^{(1)}/F$; see Example \ref{eg:fieldK1}. For $\QQ_p/\ZZ_p$ over a field $\calL$, denote the classical Selmer group by $\Sel_{\cl}(\QQ_p/\ZZ_p/\calL)$. In \cite[Theorem B]{fujii2017on}, under some conditions, he proved that if $\Sel_{\cl}(\QQ_p/\ZZ_p/F^{(1)})$ is pseudonull over $\ZZ_p\lrbracket{\Gal(F^{(1)}/F)}$, then $\Sel_{\cl}(\QQ_p/\ZZ_p/\tilF)$ is pseudonull over $\ZZ_p\lrbracket{\Gal(\tilF/F)}$.

This theorem of Fujii may be viewed as a propagation result for the pseudonullity of certain Iwasawa modules, passing from a special $\ZZ_p$-extension of $F$ to the compositum of all $\ZZ_p$-extensions of $F$. In this sense, it may be regarded as a prototype of Conjecture~\ref{myconj:Bplus}.

\subsection{Main results and tools}
We now introduce the main results of our article. To establish our control theorems, we often impose the following \emph{Greenberg finiteness condition}:
\begin{condition}
    \item[(RF)] $A(\sfL)$ is a cofinitely generated \emph{cotorsion} $R$-module for all finite extensions $\sfL/\sfF_{\infty}$ inside $\widetilde{\sfF}$, where $\sfF = F$ or $F_{v}$ for a localization of $F$ at a place $v$ above $p$, $\sfF_{\infty}/\sfF$ is a $\ZZ_p$-extension of $\sfF$ and $\widetilde{\sfF}/\sfF$ is a $\ZZ_{p}^{d}$-extension of $\sfF$ containing $\sfF_{\infty}$ with $d \geq 2$. \label{eq:rf0}
\end{condition}
We say a $\ZZ_p$-extension $F_{\infty}/F$ is \emph{cyclotomic-like} if every place of $F$ is finitely decomposed in $F_{\infty}$.

\begin{maintheorem}[{Theorem \ref{thm:maintheorem}}: Conjecture B]\label{main:I}
Let $F$ be a number field. Suppose $F_{\infty}/F$ is a cyclotomic-like $\ZZ_p$-extension, and let $\tilF/F$ be a $\ZZ_p^{d}$-extension with $d>1$ such that $\tilF$ contains $F_{\infty}$. If condition \dref{eq:rf0} holds for $F_{\infty}/F$, then we have the following results.
\begin{enumerate}[label = \rm (\arabic*)]
    \item \emph{(Conjecture \ref{myconj:Bminus})} If $Y(A/F_{\infty})$ is a finitely generated $R$-module, then $Y(A/\tilF)$ is a finitely generated $R\lrbracket{\Gal(\tilF/F_{\infty})}$-module.
    \item \emph{(Conjecture \ref{myconj:Bplus})} Assume further that $R$ is a Cohen--Macaulay ring. If $Y(A/F_{\infty})$ is a pseudonull $R\lrbracket{\Gal(F_{\infty}/F)}$-module, then $Y(A/\tilF)$ is a pseudonull  $R\lrbracket{\Gal(\tilF/F)}$-module.
\end{enumerate}
\end{maintheorem}

The key methods for establishing Main Theorem \ref{main:I} are twofold. The first is the following \emph{vertical control theorem}.

\begin{maintheorem}[{Theorem \ref{thm:verticalcontrol}}: Vertical control theorem] \label{main:II}

Assume that $F_{\infty}/F$ is cyclotomic-like. Assume that the condition \dref{eq:rf0} holds. Then the kernel and the cokernel of the Pontryagin dual $\res_{\tilF/F_{\infty}}^{\vee}$ of the restriction map
\[
\res_{\tilF/F_{\infty}}: \Sel(A/F_{\infty}) \rightarrow \Sel(A/\tilF)^{\Gal(\tilF/F_{\infty})},
\]
are both finitely generated torsion $R$-modules. In particular, as a morphism of $R\lrbracket{\Gal(F_{\infty}/F)}$-modules, $\res_{\tilF/F_{\infty}}^{\vee}$ is a pseudo-isomorphism.
\end{maintheorem}

The key technique in proving Main Theorem \ref{main:II} is a method of Ralph Greenberg for bounding the kernel and cokernel of the restriction map $\res_{\tilF/F_{\infty}}$. This method will be introduced in Section \ref{sec:greenbergmethod}.

Second, we establish a \emph{coinvariant criterion for pseudonullity} as follows.

\begin{maintheorem}[Theorem {\ref{catenary}}: A coinvariant criterion for pseudonullity]\label{main:III}
    Let $\Lambda$ be a commutative Noetherian profinite ring and assume that $\Lambda$ is Cohen--Macaulay. Let $T$ be a regular element in $\Lambda$ and assume $T$ is in the Jacobson radical of $\Lambda$. Let $M$ be a finitely generated $\Lambda$-module. If $M/TM$ is a pseudonull $\Lambda/T\Lambda$-module, then $M$ is a pseudonull $\Lambda$-module.
\end{maintheorem}
This is a vast generalization of \cite[Lemma 7]{fujii2017on} which treats only the case $R = \ZZ_p$ with $\Lambda = \ZZ_p\lrbracket{T_1, \ldots, T_{d}}$ a multivariable Iwasawa algebra over $\ZZ_p$. In fact, Main Theorem \ref{main:III} holds under the strictly weaker condition that $\Lambda$ is Noetherian, catenary, and locally equidimensional. To avoid interrupting the flow of the article, we postpone this result to Theorem \ref{catenary2} in the appendix, where we also include several counterexamples indicating that each of these assumptions is indispensable.

In our article, we apply Main Theorem \ref{main:III} with Cohen--Macaulay rings $R$ as coefficient rings and with $\Lambda$ a multivariable Iwasawa algebra $R\lrbracket{T_1, \ldots, T_d}$ over $R$. This will be stated as Theorem \ref{thm:descentpseudonull}.

\subsection{Application of main results}
To apply Main Theorems \ref{main:I} and \ref{main:II}, it suffices to check whether the Galois representation $T$ over $R$ satisfies finiteness condition \dref{eq:rf0}. In Section \ref{sec:applications}, we verify the condition in the following examples, mostly when $F_{\infty}/F$ is the cyclotomic $\ZZ_p$-extension $F_{\cyc}/F$:

(1) \emph{Tate motive}: $A = \QQ_p/\ZZ_p(1)$. (Theorem \ref{thm:apptate})

(2) \emph{Elliptic curves}: $A$ is given by the Tate module of elliptic curves over number fields $F$:
    \begin{itemize}
        \item with potential good reduction at all places of $F$ above $p$ (Theorem \ref{thm:appec}), or
        \item with potential multiplicative reduction at all places of $F$ above $p$, when $F_{v} = \QQ_p$ for all such places $v$ for all \footnote{If $E$ already has multiplicative reduction at all places of $F$ above $p$, then we allow $p = 3$.} $p \geq 5$ (Theorems \ref{thm:appecmultcyc} and \ref{thm:appecmultcyc_nonsplit}).
    \end{itemize}
    Taken together, these results cover almost all elliptic curves over $F$ (subject to the additional conditions \dref{item:splitminus} and \dref{item:large} in the text); that is, they allow any reduction type at $p$, including mixed reduction types at the places of $F$ above $p$. (Theorem \ref{thm:appecoverF})

(3) \emph{Classical cuspidal newforms}: $A$ is given by the Galois representation attached to a cuspidal eigenform $f \in S_{k}(\Gamma_{0}(N))$ with weight $k \geq 2$ even and $p \nmid N$. (Theorem \ref{thm:appmodforms})

(4) \emph{Cuspidal Hida families}: $A$ is given by the Galois representation attached to an $\II$-adic cuspidal Hida family $\bff$ passing through a cuspidal newform $f$ with even weight and residually irreducible Galois representation, under the assumption that $\II$ is Cohen--Macaulay (see Theorem \ref{thm:apphida}, and also Remark \ref{rem:branchCM} for the Cohen--Macaulay condition).

\begin{remark}
    When $A = \QQ_p/\ZZ_p$, it is easy to see that condition \dref{eq:rf0} is violated. For the cyclotomic $\ZZ_p$-extension, Proposition \ref{prop:counter} gives a case in which the vertical control theorem (Main Theorem \ref{main:II}) fails. However, we still expect that if one considers the $\ZZ_p$-extension $F^{(1)}/F$, the vertical control theorem (Main Theorem \ref{main:II}) holds for $A = \QQ_p/\ZZ_p$.
\end{remark}

\subsection{Notations}
Throughout this article, let $p$ be an odd prime number. For any field $F$, we denote its absolute Galois group by $\Gal_{F}$. For a number field $F$ (i.e. a finite extension of $\QQ$) and a finite set $S$ of places of $F$ containing the places above $p$ and the infinite places, we denote by $F_{S}$ the maximal algebraic extension of $F$ that is unramified outside $S$. For any algebraic (possibly infinite) extension $\calL$ of $F$ contained in $F_{S}$, we write $\Gal_{\calL, S} := \Gal(F_{S}/\calL)$. We write $S_{\calL}$ for the set of places of $\calL$ lying over those in $S$.

For any commutative ring $R$ and any ideal $I$ of $R$, we use $\Ht_{R}(I)$ to denote the height of $I$ in $R$. If $R$ is a local ring with maximal ideal $\frm$, then the $\frm$-depth of $R$ as an $R$-module (following \cite[\href{https://stacks.math.columbia.edu/tag/00LE}{Tag 00LE}]{stacks-project}) is called the \emph{depth} of $R$, denoted by $\depth(R)$. 

In this article, most lemmas hold when $R$ is a commutative Noetherian profinite $\ZZ_p$-algebra. For applications, we may impose additional conditions on $R$. For example, we may assume that $R$ is catenary, satisfies Serre's condition $(S_1)$, is Cohen--Macaulay, or is Gorenstein. 
These definitions will be recalled in the text.

If $M$ is a pro-$p$ group or a discrete $p$-primary group, we denote the Pontryagin dual of $M$ by $M^{\vee} := \Hom_{\cts}(M, \QQ_p/\ZZ_p)$.

\subsection*{Acknowledgement} The authors have benefited greatly from insightful discussions with Haidong Li, Meng Fai Lim, Preston Wake, Xin Wan, and Mulun Yin, and are sincerely grateful to them. They also thank David Loeffler for his valuable answer to their questions about Kuga--Sato varieties in the following MathStackExchange post:
\begin{center}
\url{https://math.stackexchange.com/questions/5122065/}.
\end{center}
The counterexamples in Section \ref{sec:appA} were found with the assistance of either \texttt{Rethlas} \cite{rethlas}, a natural-language reasoning system for mathematics accessed via \texttt{Codex CLI}, or \texttt{GPT-5.6 Sol (High)}, accessed via the \texttt{ChatGPT} desktop application. All counterexamples were subsequently verified independently by the authors. The authors gratefully acknowledge this assistance and thank the \texttt{Rethlas} development team for making the system widely available to the mathematical community.

\section{Preliminaries}
In this section, we review the preliminaries for this article, including the definitions of fine Selmer groups and cyclotomic-like $\ZZ_p$-extensions.

\subsection{Fine Selmer groups}

Let $R$ be a profinite $\ZZ_p$-algebra and assume $R$ is Noetherian and commutative. Let $T$ be a finitely generated compact $R$-module. 
We further assume that $T$ has a continuous $R$-linear $\Gal_{F}$-action unramified outside $S$. This group action on $T$ induces a continuous group homomorphism 
\[
\rho: \Gal_{F,S} \rightarrow \Aut_{R}(T).
\]
One may think of $R$ as a Galois deformation ring and of $T$ as the representation associated with $R$. 

We write \footnote{Compared with \cite{lim2017}, the fine Selmer group is defined for $W := T^{\vee}(1)$ in \textit{op.cit.}.} $A := T \otimes_{R} \Hom_{\ZZ_p}(R, \QQ_p/\ZZ_p)$. Let $v$ be a prime in $S$. For each \emph{finite} extension $L$ of $F$ contained in $F_{S}$, we define
\[
\rmK^{i}_{v}(A/L) := \bigoplus_{w \mid v} \rmH^{i}(L_w, A)
\]
for $i = 0, 1$, where $w$ runs over the finite set of places of $L$ above $v$. If $\calL$ is an infinite extension of $F$ inside $F_{S}$, we define
\[
\rmK^{i}_v(A/\calL) := \varinjlim_{L} \rmK_{v}^{i}(A/L),
\]
where the direct limit is taken over all finite extensions $L$ of $F$ contained in $\calL$ under the restriction maps.

\begin{definition}
For any algebraic (possibly infinite) extension $\calL$ of $F$ contained in $F_{S}$, the \emph{fine Selmer group} of $A$ over $\calL$ with respect to $S$ is defined to be
\[
\Sel_{S}(A/\calL) := \ker\left(\rmH^{1}(\Gal_{\calL, S}, A) \rightarrow \bigoplus_{v \in S} \rmK^{1}_v(A/\calL)\right).
\]
We will write $Y_{S}(A/\calL)$ for the Pontryagin dual $\Sel_{S}(A/\calL)^{\vee}$ of the Selmer group, called the \emph{dual fine Selmer group} of $A$ over $\calL$.
\end{definition}

It follows from the Poitou--Tate exact sequence that
\begin{equation} \label{eq:poitoutate}
    0 \rightarrow Y_{S}(A/\calL) \rightarrow \varprojlim_{L} \rmH^{2}(\Gal_{L,S}, T^{\ast}) \rightarrow \left( \bigoplus_{v \in S} \rmK^{0}_{v}(A/\calL) \right)^{\vee}
\end{equation}
where $T^{\ast} := \Hom_{R}(T, R(1))$. We shall denote 
\[
\rmH^{2}_{S}(\calL/F, T^{\ast}) := \varprojlim_{L} \rmH^{2}(\Gal_{L,S}, T^{\ast}).
\]
To simplify notation, from now on we shall write 
\[
A(\calL) := \rmH^{0}(\Gal_{\calL,S}, A), \text{ and } A(\calL_{w}) := \rmH^{0}(\Gal_{\calL_{w}}, A)
\]
for any place $w$ of $\calL$.

\subsection{Cyclotomic-like $\ZZ_p$-extensions}
\begin{definition}
Let $F$ be a number field and $F_{\infty}/F$ be a $\ZZ_p$-extension. We say $F_{\infty}/F$ is \emph{cyclotomic-like} if every place of $F$ decomposes finitely in $F_{\infty}/F$.
\end{definition}

\begin{example}
    The cyclotomic $\ZZ_p$-extension $F_{\cyc}/F$ of any number field $F$ is cyclotomic-like.
\end{example}

\begin{example} \label{eg:fieldK1}
    Let $F/F^{+}$ be a CM field extension, i.e. $F^{+}$ is a totally real number field, and $F$ is a totally imaginary quadratic extension of $F^{+}$. Let $S_{F^{+}} := \{\frp_1, \ldots, \frp_s\}$ be the set of primes of $F$ above $p$. We assume that each of the primes above $p$ in $F^{+}$ splits in $F$. Write $\frp_i \calO_{F} = \frP_i \overline{\frP_i}$ for $1 \leq i \leq s$, where $\overline{\frP_{i}}$ is the complex conjugation of $\frP_i$, and set $S := \{\frP_1, \ldots, \frP_s\}$. It follows from \cite[Theorem 2.1]{peikai2024} that the $\ZZ_p$-extension $F_{\infty}/F$ unramified outside $S$ exists, and it is unique if we assume that Leopoldt's conjecture holds for $F$. This $\ZZ_p$-extension is called the \emph{$S$-ramified $\ZZ_p$-extension}, denoted by $F^{(1)}/F$. It follows from \cite[Lemma 3]{fujii2017on} that $F^{(1)}/F$ is a cyclotomic-like $\ZZ_p$-extension if $p$ splits completely in $F$. As expected in \cite{peikai2024}, such $\ZZ_p$-extensions for CM fields have similar properties to the cyclotomic $\ZZ_p$-extensions of a totally real field.
\end{example}

Recall that for a general algebraic extension $\calL/F$ contained in $F_{S}$, the fine Selmer group $\Sel_{S}(A/\calL)$ may depend on the choice of $S$. We can prove that it does not when $\calL/F$ contains a cyclotomic-like $\ZZ_p$-extension $F_{\infty}/F$, generalizing \cite[Lemma 3.2]{sujatha2018} from the case $F_{\infty} = F_{\cyc}$.

In general, we expect that Conjectures \ref{myconj:A}, \ref{myconj:Bminus}, \ref{myconj:B} and \ref{myconj:Bplus} hold when replacing the cyclotomic $\ZZ_p$-extension $F_{\cyc}/F$ by any cyclotomic-like $\ZZ_p$-extension $F_{\infty}/F$.

\begin{proposition} \label{prop:notdependonS}
Suppose $F_{\infty}/F$ is a cyclotomic-like $\ZZ_p$-extension. Let $S \subseteq T$ be two finite sets of places of $F$ containing those above $p$ and the ramified places for $A$. Then $\Sel_{S}(A/\calL) = \Sel_{T}(A/\calL)$ for any subextension $\calL/F$ of $F_{S}/F$ containing $F_{\infty}/F$. 
\end{proposition}

\begin{proof}
Note that 
\[
\Sel_{S}(A/\calL) = \varinjlim_{L/F} \Sel_{S}(A/L F_{\infty})
\]
with $L/F$ running through all finite subextensions of $\calL/F$, we may assume that $\calL = F_{\infty}$. The remaining proof follows in the same way as \cite[Lemma 3.2]{sujatha2018}, which we recall here for the convenience of readers.

Let $v$ be any place of $F_{\infty}$ lying over $l$ with $l\neq p$, and denote by $I_{v}$ the inertia group at $v$ in $\Bar{F}/F_{\infty}$ and by $k_v$ the residue field of $F_\infty$ at $v$. Let $G_{k_v}$ be the absolute Galois group of $k_v$. The wild inertia subgroup $P_v$ of $I_v$ is pro-$l$. The tame quotient is 
\[I_v/P_v\cong \prod_{q\neq l}\ZZ_q(1), 
\]
where the $(1)$ indicates that $G_{k_v}$ acts through the Tate twist by 1. Hence the pro-$p$ part of $I_v$ is isomorphic to $\ZZ_p$ and $I_{v}$ has $p$-cohomological dimension $1$. We have $\rmH^{q}(I_v, A)=0$ for all $q > 1$. Applying the Hochschild-Serre spectral sequence
\[
E_{2}^{p,q} = \rmH^{p}(k_{v}, \rmH^{q}(I_v, A)) \Rightarrow \rmH^{p+q}(F_{\infty,v}, A),
\]
by \cite[Exercise 4 in Section 2.4]{MR1737196}, we have the following exact sequence
\begin{align*}
0 &\rightarrow \rmH^{1}(k_v, A^{I_v}) \rightarrow \rmH^{1}(F_{\infty,v}, A) \rightarrow \rmH^{0}(k_v, \rmH^{1}(I_v, A)) \\
&\rightarrow \rmH^{2}(k_v, A^{I_v}) \rightarrow \rmH^{2}(F_{\infty,v}, A) \rightarrow \rmH^{1}(k_v, \rmH^{1}(I_v, A))  \\
&\rightarrow \rmH^{3}(k_v, A^{I_v}) \rightarrow \cdots.
\end{align*}
Since any prime of $F$ splits finitely in $F_\infty$, the decomposition group of $v$ in $F_\infty/F$ is nontrivial. Hence, 
\[G_{k_v}\cong \prod_{q\neq p}\ZZ_q.\]   
Thus, $G_{k_v}$ has no pro-$p$ quotient, which implies $\rmH^i(k_v, M)=0$ for $i\geq 1$ and any pro-$p$ module $M$. Hence,
\begin{equation} \label{eq:hsspecseq}
\rmH^{1}(F_{\infty,v}, A) \simeq H^{0}(k_v, \rmH^{1}(I_v, A)), \quad \rmH^{2}(F_{\infty,v}, A) = 0.    
\end{equation}
Additionally, if $v$ is not in $S_{F_{\infty}}$, then $I_v$ acts trivially on $A$ because $S$ already contains all ramified places for $A$. We have 
\begin{equation} \label{eq:hsspecseq2}
    \rmH^{1}(I_v, A)\simeq \Hom(\ZZ_p(1), A) \simeq A(-1)
\end{equation}

As $T$ contains $S$, we have the following Gysin exact sequence
\begin{align} \label{eq:gysin}
0 &\rightarrow \rmH^{1}(\Gal_{F_{\infty},S}, A) \rightarrow \rmH^{1}(\Gal_{F_{\infty},T}, A) \rightarrow \bigoplus_{v \in (T \smallsetminus S)_{F_{\infty}}} \rmH^{0}(k_v, A(-1)) \notag \\
&\rightarrow \rmH^{2}(\Gal_{F_{\infty},S}, A) \rightarrow \rmH^{2}(\Gal_{F_{\infty},T}, A) \rightarrow \bigoplus_{v \in (T \smallsetminus S)_{F_{\infty}}} \rmH^{1}(k_v, A(-1))=0
\end{align}
Therefore,
\begin{align*}
  \Sel_{T}&( A /F_{\infty}) \\&= \ker \left(\rmH^{1}(\Gal_{F_{\infty}, T}, A) \rightarrow \bigoplus_{v \in T_{F_{\infty}}} \rmH^{1}(F_{\infty,v}, A)\right) \\
    &= \ker \left(\rmH^{1}(\Gal_{F_{\infty}, T}, A) \rightarrow \bigoplus_{v \in (T \smallsetminus S)_{F_{\infty}}} \rmH^{1}(F_{\infty,v}, A) \oplus \bigoplus_{v \in S_{F_{\infty}}} \rmH^{1}(F_{\infty,v}, A)\right) \\
    &\xlongequal{\eqref{eq:hsspecseq}\eqref{eq:hsspecseq2}} \ker \left(\rmH^{1}(\Gal_{F_{\infty}, T}, A) \rightarrow \bigoplus_{v \in (T \smallsetminus S)_{F_{\infty}}} \rmH^{0}(k_v, A(-1)) \times \bigoplus_{v \in S_{F_{\infty}}} \rmH^{1}(F_{\infty,v}, A)\right) \\
    &\xlongequal{\eqref{eq:gysin}} \ker \left(\rmH^{1}(\Gal_{F_{\infty}, S}, A) \rightarrow \bigoplus_{v \in S_{F_{\infty}}} \rmH^{1}(F_{\infty,v}, A)\right) \\
    &= \Sel_{S}(A/F_{\infty}),
\end{align*}
as desired.

\end{proof}

For any algebraic extension $\calL$ of $F$ inside $F_{S}$ that contains a cyclotomic-like $\ZZ_p$-extension of $F$, we apply Proposition \ref{prop:notdependonS}; thus, we may drop the subscript $S$ from the notation $\Sel_{S}(A/\calL)$ and $Y_{S}(A/\calL)$ for simplicity.

\section{Greenberg's method} \label{sec:greenbergmethod}
The main tool in establishing the vertical control theorem \ref{thm:verticalcontrol} is our generalization of a method of R.~Greenberg appearing in the proof of \cite[Proposition 3.3]{greenberg2003galois}. The main theorem in this section is Theorem \ref{thm:generalgreenberg}.

Let $\widetilde{\sfF} / \sfF_{\infty} / \sfF$ be a tower of Galois extensions with
\[
\Gal(\sfF_{\infty}/\sfF) = \ZZ_p, \, \Gal(\widetilde{\sfF}/\sfF) = \ZZ_p^{n+1} \text{ for } n \geq 1.
\]
Here $\sfF$ is a finite extension of $\QQ$ or $\QQ_p$. To help readers better understand Greenberg's method, we start with a review of Greenberg's original setup where the coefficient ring $R$ is $\ZZ_p$.

\subsection{The baby case: $R = \ZZ_p$}
We first consider the case where $R = \ZZ_p$ and $T$ is a finitely generated \emph{free} $\ZZ_p$-module. In this case, $A = T \otimes_{\ZZ_p} \QQ_p/\ZZ_p$ and $T$ is the Tate module of $A$.

\begin{theorem}[Greenberg's method for $R = \ZZ_p$] \label{thm:greenberg}
With the notation as above, suppose that $A(\sfL)$ is finite for all finite extensions $\sfL/\sfF_{\infty}$ inside $\widetilde{\sfF}$. Then
\[
\rmH^{i}(\widetilde{\sfF}/\sfF_{\infty}, A(\widetilde{\sfF})) \text{ is finite}
\]
for $i \geq 1$.
\end{theorem}

\begin{proof}
We start with some preliminary observations \footnote{This is taken from the arguments in \cite[page 12-13]{imai1975} in the case where $T = T_p(A)$ is the Tate module of an elliptic curve $A$ over $\QQ$.} Let $\sfL / \sfF$ be any algebraic extension of $\sfF$, and consider $T(\sfL) := T^{\Gal_{\sfL}}$. By the elementary divisor theorem, under a suitable $\ZZ_p$-basis, we can write
\[
T(\sfL) = p^{a_1} \ZZ_p \oplus \cdots p^{a_r} \ZZ_p \oplus 0 \oplus \cdots \oplus 0 \subseteq T \simeq \ZZ_p^{n},
\]
for some nonnegative integers $a_i$ for $1 \leq i \leq r$. We claim  that $a_i = 0$ for all $1 \leq i \leq r$. In fact, it suffices to check that for any $\sigma \in \Aut_{\ZZ_p}(T)$, $x \in T$ and any integer $a \geq 0$, we have the implication
\[
\sigma(p^{a} x) = p^{a}x \Rightarrow \sigma(x) = x, 
\]
which holds since $T$ is torsion-free. This implies that $T(\sfL)$ is simply a $\ZZ_p$-direct summand of $T$. We define $V(\sfL) := T(\sfL) \otimes_{\ZZ_p} \QQ_p$, and then
\[
A(\sfL)_{\ddiv} := V(\sfL)/T(\sfL), \text{ and } A(\sfL) \simeq A(\sfL)_{\ddiv} \times A(\sfL)_{\tors},
\]
where $A(\sfL)_{\ddiv} \simeq (\QQ_p/\ZZ_p)^{r}$ is the divisible part of $A(\sfL)$ and $A(\sfL)_{\tors}$ is the torsion part of $A(\sfL)$, which has finite cardinality.

Moreover, we have the equivalences between the following statements:
\begin{enumerate}[label = (\roman*)]
    \item $A(\sfL)$ is an infinite group, \label{item:equivi}
    \item for any positive integer $n$, there exists an element $x_n(\sfL)$ of order $p^n$,
    \item $T(\sfL) \neq 0$.
    \item $V(\sfL) := V^{\Gal_{\sfL}} \neq 0$. \label{item:equiviv}
\end{enumerate}
To see the second equivalence, we consider the projective system consisting of sets
\[
A_{n}(\sfL) := \{x \in A(\sfL): x \text{ is of order } p^n \},
\]
and the maps $\pr_{n}: A_{n}(\sfL) \rightarrow A_{n-1}(\sfL)$ which are induced from multiplication by $p$. Since the projective limit of nonempty finite sets is nonempty, the second statement implies the third, and the converse is trivial. For the third equivalence, note that we have seen that $T(\sfL)$ is a free $\ZZ_p$-module of finite rank.

A key observation of Greenberg is that, under the condition that $A(\sfL)$ is finite for all finite extensions $\sfL/\sfF_{\infty}$, there is an intermediate field $\sfF_{\circ}$ in $\widetilde{\sfF}/\sfF_{\infty}$ such that
\begin{enumerate}[label = (\arabic*)]
    \item $V(\widetilde{\sfF})^{\Gal(\widetilde{\sfF}/\sfF_{\circ})} = 0$, \label{item:circ1}
    \item $Z := \Gal(\widetilde{\sfF}/\sfF_{\circ})$ is isomorphic to $\ZZ_p$. \label{item:circ2}
\end{enumerate}
Granting the existence of such an intermediate field $\sfF_{\circ}$, we show inductively, using the inflation-restriction exact sequence, that the cohomology groups $\rmH^{i}(\widetilde{\sfF}/\sfF_{\circ}, V(\widetilde{\sfF}))$ vanish for $i \geq 1$. First, the inflation-restriction exact sequence gives
\[
0 \rightarrow \rmH^{1}(\sfF_{\circ}/\sfF_{\infty}, V(\widetilde{\sfF})^{Z}) \rightarrow \rmH^{1}(\widetilde{\sfF}/\sfF_{\infty}, V(\widetilde{\sfF})) \rightarrow \rmH^{1}(Z, V(\widetilde{\sfF}))^{\Gal(\sfF_{\circ}/\sfF_{\infty})}.
\]
By the defining property (1) of $\sfF_{\circ}$, the first term $\rmH^{1}(\sfF_{\circ}/\sfF_{\infty}, V(\widetilde{\sfF})^{Z})$ vanishes, and hence $\rmH^{1}(\widetilde{\sfF}/\sfF_{\infty}, V(\widetilde{\sfF}))$ injects into $\rmH^{1}(Z, V(\widetilde{\sfF}))$. By the defining property (2) of $\sfF_{\circ}$, choose $z$ as a topological generator of $Z$. Then
\[
\rmH^{1}(\widetilde{\sfF}/\sfF_{\infty}, V(\widetilde{\sfF})) \hookrightarrow \rmH^{1}(Z, V(\widetilde{\sfF}) )= V(\widetilde{\sfF}) / (z-1) V(\widetilde{\sfF}),
\]
with the rightmost term being zero by (1). Therefore, \fbox{$\rmH^{1}(\widetilde{\sfF}/\sfF_{\infty}, V(\widetilde{\sfF})) = 0$}. This allows us to consider the inflation-restriction exact sequence in degree $2$, that is,
\[
0 \rightarrow \rmH^{2}(\sfF_{\circ}/\sfF_{\infty}, V(\widetilde{\sfF})^{Z}) \rightarrow \rmH^{2}(\widetilde{\sfF}/\sfF_{\infty}, V(\widetilde{\sfF})) \rightarrow \rmH^{2}(Z, V(\widetilde{\sfF}))^{\Gal(\sfF_{\circ}/\sfF_{\infty})}.
\]
By (1), the leftmost term is zero, and by (2), the rightmost term is zero since $Z$ has $p$-cohomological dimension $1$. Therefore, \fbox{$\rmH^{2}(\widetilde{\sfF}/\sfF_{\infty}, V(\widetilde{\sfF})) = 0$}. This further allows us to use the inflation-restriction exact sequence in degree $3$. Repeating the argument, we obtain that
\begin{center}
\fbox{$\rmH^{i}(\widetilde{\sfF}/\sfF_{\infty}, V(\widetilde{\sfF})) = 0$ for all $i \geq 1$.}   
\end{center}

Having shown this, the result follows from dimension shifting. Indeed, consider the short exact sequence
\[
0 \rightarrow T(\widetilde{\sfF}) \rightarrow V(\widetilde{\sfF}) \rightarrow A(\widetilde{\sfF})_{\ddiv} \rightarrow 0,
\]
we obtain
\[
\rmH^{i}(\widetilde{\sfF}/\sfF_{\infty}, A(\widetilde{\sfF})_{\ddiv}) \simeq \rmH^{i+1}(\widetilde{\sfF}/\sfF_{\infty}, T(\widetilde{\sfF}))
\]
for $i \geq 1$. Since $T(\widetilde{\sfF})$ is a $\ZZ_p$-module of finite rank, it is well known that $\rmH^{j}(\widetilde{\sfF}/\sfF_{\infty}, T(\widetilde{\sfF}))$ is a finitely generated $\ZZ_p$-module (see, for example, \cite[page 513]{huber2010}) for $j \geq 0$, and 
\[
\rank_{\ZZ_p} \rmH^{j}(\widetilde{\sfF}/\sfF_{\infty}, T(\widetilde{\sfF})) = \dim_{\QQ_p} \rmH^{j}(\widetilde{\sfF}/\sfF_{\infty}, V(\widetilde{\sfF})). 
\]
We have shown that the latter is zero for $j \geq 0$, and hence $\rmH^{i+1}(\widetilde{\sfF}/\sfF_{\infty}, T(\widetilde{\sfF}))$ is finite for all $i \geq 1$. Consequently, $\rmH^{i}(\widetilde{\sfF}/\sfF_{\infty}, A(\widetilde{\sfF})_{\ddiv})$ is finite for $i \geq 1$. Since $A(\widetilde{\sfF})_{\tors}$ is finite, it follows that $\rmH^{i}(\widetilde{\sfF}/\sfF_{\infty}, A(\widetilde{\sfF}))$ is finite for $i \geq 1$.

The existence of an intermediate field $\sfF_{\circ}$ of $\widetilde{\sfF}/\sfF_{\infty}$ satisfying \ref{item:circ1} and \ref{item:circ2} follows from the last paragraph of the proof of \cite[Proposition 3.3]{greenberg2003galois}, which we recall here.

Consider $V(\widetilde{\sfF})$ as a representation space for $N := \Gal(\widetilde{\sfF}/\sfF_{\infty})$. Since $N$ is \emph{abelian}, all of its irreducible representations over $\barQQ_p$ are one-dimensional. Thus the composition factors in the representation space $V(\widetilde{\sfF}) \otimes_{\QQ_p} \barQQ_p$ are one-dimensional, and the action of $N$ on them is given by homomorphisms $\chi_i: N \rightarrow \barQQ_p^{\times}$ for $1 \leq i \leq r$. Since $\rmH^{0}(\calN, V(\widetilde{\sfF})) = 0$ for every subgroup $\calN$ of finite index in $N$ (recall the equivalence between \ref{item:equiviv} and \ref{item:equivi}, with \ref{item:equivi} being the assumption of the theorem), it is clear that $\chi_i|_{\calN}$ is nontrivial for each $i$. Indeed, otherwise the elements of $\calN$ would have $1$ as a common eigenvalue, and a common eigenvector would exist since $\calN$ is abelian. Thus, each $\chi_i$ has infinite order. One simply chooses $Z$ isomorphic to $\ZZ_p$ so that $Z \otimes_{\ZZ_p} \QQ_p$ is not contained in any of the proper subspaces $\ker(\chi_i) \otimes_{\ZZ_p} \QQ_p$ of $N \otimes_{\ZZ_p} \QQ_p$, which is certainly possible. The intermediate field $\sfF_{\circ}$ is then taken to be the fixed field of $\widetilde{\sfF}$ under $Z$.
\end{proof}

We emphasize that even if $\Gal(\widetilde{\sfF}/\sfF_{\infty}) \simeq \ZZ_p^{2}$, our proof does not automatically imply that $\sfF_{\circ} = \sfF_{\infty}$, but only that $\sfF_{\circ}/\sfF_{\infty}$ is finite of degree $p^{t}$ for some $t \geq 0$. This finiteness is not used in the article.

\subsection{General coefficients} \label{sec:generalcoefficients}
In this section, we generalize Theorem \ref{thm:greenberg} to Theorem \ref{thm:generalgreenberg} for a general coefficient ring $R$. Recall that $R$ is a profinite $\ZZ_p$-algebra and that $R$ is Noetherian and commutative. In this case, for any $\ZZ_p^d$-extension $F_{\infty}/F$, the Iwasawa algebra $R\lrbracket{\Gal(F_{\infty}/F)}\cong R\lrbracket{T_1,T_2,\cdots,T_d}$ is a Noetherian ring.

Let $M$ be an $R$-module. Recall that an element of $M$ is called \emph{torsion} if it becomes zero after being multiplied by a nonzero divisor of $R$. The module $M$ is called a \emph{torsion} $R$-module if every element of $M$ is a torsion element. We say that $M$ is a \emph{cotorsion} $R$-module if its Pontryagin dual $M^\vee$ is a torsion $R$-module.

Let $R$ be a commutative Noetherian profinite ring. Let $M$ be a finitely generated torsion $R$-module. Let $\Ht(\frp)$ denote the height of the ideal $\frp$ in $R$. Recall that $M$ is a \emph{pseudonull} $R$-module if the localization $M_\frp=0$ for any prime ideal $\frp$ in $R$ with height $\Ht(\frp)\leq 1$. In other words, $M$ is a pseudonull $R$-module if and only if the height of the annihilator $\Ht(\Ann_R(M))\geq 2$. 

\subsubsection{A torsion criterion for pseudonullity}
A finitely generated compact $\ZZ_p\lrbracket{T}$-module $M$ is pseudo-null if and only if $M$ is a finite module. In other words, $M$ is a finitely generated torsion $\ZZ_p$-module. This motivates the following proposition.
\begin{proposition} \label{prop:torsionpseudonull}
    Assume $R$ is a Noetherian commutative profinite ring. Let $H$ be a closed normal subgroup of $G\cong \ZZ_p^d$ such that $G/H\cong\ZZ_p$.  Let $M$ be a compact $R\lrbracket{G}$-module which is finitely generated over $R\lrbracket{H}$. If $M$ is a torsion $R\lrbracket{H}$-module then $M$ is a pseudonull $R\lrbracket{G}$-module.
\end{proposition}

\begin{proof}
    By choosing a suitable basis of $G \simeq \ZZ_p^{d}$, we can write $R\lrbracket{H} = \Lambda$ and $R\lrbracket{G} = \Lambda\lrbracket{T}$, where $T$ is an indeterminate. Since $M$ is a finitely generated $\Lambda$-module and $TM \subset M$, the Cayley--Hamilton theorem (see \cite[Proposition 2.4]{atiyah1969introduction}) yields a monic polynomial $f(T) \in \Lambda[T]$ such that $f(T)M = 0$.

    To prove $M$ is a pseudonull $\Lambda\lrbracket{T}$-module, we shall prove that for any prime ideal $\frP$ of $\Lambda\lrbracket{T}$ with height $\Ht_{\Lambda\lrbracket{T}}(\frP) \leq 1$, the module $M_{\frP} = 0$. 
    
    If $f(T) \not\in \frP$, then clearly $M_{\frP} = 0$. We therefore consider the case $f(T) \in \frP$. Set $\frp = \frP \cap \Lambda$. By \cite[Exercise 5.5, Proposition 10.4]{atiyah1969introduction}, $\Lambda\lrbracket{T}$ is a flat $\Lambda$-algebra. It follows from the going-down property of flat morphisms (see \cite[(13.B) Theorem 19(2)]{MR575344}) that 
    \begin{equation} \label{eq:flatheight}
        1 \geq \Ht_{\Lambda\lrbracket{T}}(\frP) = \Ht_{\Lambda}(\frp) + \Ht_{\Lambda\lrbracket{T}/\frp}(\frP/\frp\Lambda\lrbracket{T}).    
    \end{equation}
    Since $f(T)$ is a monic polynomial, the image of $f(T)$ in $\Lambda\lrbracket{T}/\frp \Lambda\lrbracket{T}$ is not a zero divisor. Since $f(T) \in \frP$, Krull's principal ideal theorem (see \cite[Corollary 11.17]{atiyah1969introduction}) gives
    \[
    \Ht_{\Lambda\lrbracket{T}/\frp}(\frP/\frp\Lambda\lrbracket{T}) \geq 1.
    \]
    Combining \eqref{eq:flatheight}, we see that $\Ht_{\Lambda\lrbracket{T}/\frp}(\frP/\frp\Lambda\lrbracket{T}) = 1$ and $\Ht_{\Lambda}(\frp) = 0$. The assumption that $M$ is a \emph{finitely generated} torsion $\Lambda$-module implies that $M \otimes_{\Lambda} \Lambda_{\frp} = 0$ as long as $\frp$ is a minimal prime ideal of $\Lambda$ (see Theorem \ref{commutative algebra} for details). We therefore have $M_{\frP} = M \otimes_{\Lambda} \Lambda\lrbracket{T}_{\frP} = 0$ since $\Lambda\lrbracket{T}_{\frP}$ is a localization of $\Lambda_{\frp}$.
\end{proof}

\begin{remark}
We remark that \cite[Lemma 5.1]{lim2017} proves Proposition \ref{prop:torsionpseudonull}, together with its converse, under the assumptions that $R$ is a complete \emph{regular} local ring with finite residue field of characteristic $p$, and that $G$ is a compact pro-$p$ $p$-adic Lie group without $p$-torsion. By contrast, in Proposition \ref{prop:torsionpseudonull} we establish only one implication and assume that $G \cong \ZZ_p^{d}$ for some $d \ge 1$, but we require only that $R$ be a commutative Noetherian profinite ring, with no regularity or completeness assumptions.
\end{remark}

\subsubsection{Two homological lemmas}
Before proving the main result (Theorem \ref{thm:generalgreenberg}), we prepare the following two lemmas.

\begin{lemma}\label{cofg or cotorsion}
    Assume $M$ is a cofinitely generated $R$-module with a continuous $N\cong \ZZ_p^d$-action. Then $\rmH^i(N,M)$ are cofinitely generated $R$-modules for $i\geq 0$. If $M$ is also a cotorsion $R$-module, then $\rmH^i(N,M)$ are cotorsion $R$-modules for $i\geq 0$. 
\end{lemma}

\begin{proof}
    Let $P_{\bullet} \rightarrow R$ be a free resolution of $R$ as an $R\lrbracket{N}$-module, such that each $P_{i}$ is a finitely generated free $R\lrbracket{N}$-module of rank $r_i$. By definition, we have
    \[
    \rmH^{i}(\widetilde{\sfF}/\sfF_{\infty}, M) \simeq \rmH^{i}(\Hom_{R\lrbracket{N}}(P_{\bullet}, M)).
    \]
    Since $P_{\bullet}$ is finitely generated free $R\lrbracket{N}$-module, we have that 
    \[
    \Hom_{R\lrbracket{N}}(P_{\bullet}, M)^\vee \simeq P_{\bullet}\otimes_{R\lrbracket{N}} M^\vee \simeq (M^{\vee})^{r_i}.
    \]
    If $M$ is a cofinitely generated (cotorsion) $R$-module, then each term of $P_{\bullet}\otimes_{R\lrbracket{N}} M^\vee$ is a finitely generated (torsion) $R$-module. Thus, after taking homology, the cohomology groups $\rmH^i(N,M)$ are cofinitely generated (cotorsion) $R$-modules for $i\geq 0$. 
\end{proof}

\begin{lemma}\label{dual of M}
    Let $M$ be a cofinitely generated $R$-module with a continuous $R$-linear $Z \simeq \ZZ_p$-action. If $\rmH^0(Z,M)$ is a cotorsion $R$-module, then $\rmH^1(Z,M)$ is a cotorsion $R$-module.
\end{lemma}

\begin{proof}
 We prove that the cotorsionness of $\rmH^0(Z,M)$ as an $R$-module implies the cotorsionness of $\rmH^1(Z,M)$ as an $R$-module. Let $\gamma$ be a generator of $Z \simeq \ZZ_p$ and let $X$ be the Pontryagin dual of $M$. We have
 \[
 0\to \rmH^0(Z,M)=M[\gamma-1]\to M \xrightarrow{\gamma-1} M\to \rmH^1(Z,M)=M/(\gamma-1)M\to 0.
 \]
 By Pontryagin duality, we have
 \[
 0\to \rmH^1(Z,M)^\vee=X[\gamma-1]\to X \xrightarrow{\gamma-1} X \to \rmH^0(Z,M)^\vee=X/(\gamma-1)X\to 0
 \]
Let $Q(R)$ be the total quotient ring of $R$. Since $\rmH^0(Z,M)^\vee$ is a torsion $R$-module, we have
 \begin{multline*}
   0\to \rmH^1(Z,M)^\vee\otimes Q(R) \to X\otimes Q(R) \xrightarrow{\gamma-1} X\otimes Q(R) \\  \rightarrow \rmH^0(Z,M)^\vee\otimes Q(R)=0.  
 \end{multline*}
It is a general fact that a surjective endomorphism of a Noetherian module is automatically injective (see \cite{145362} for a proof). Thus, $\rmH^1(Z,M)^\vee\otimes Q(R)=0$, so $\rmH^1(Z,M)$ is a cotorsion $R$-module.
\end{proof}

\subsubsection{Greenberg's method for general coefficient rings}
We are now ready to state and start the proof of Greenberg's method for general coefficients, as follows.
\begin{theorem}[Greenberg's method for general coefficient rings] \label{thm:generalgreenberg}
Suppose that $A(\sfL)$ is a cotorsion $R$-module for all finite extensions $\sfL/\sfF_\infty$ inside $\widetilde{\sfF}$. Then $\rmH^{i}(\widetilde{\sfF}/\sfF_{\infty}, A(\widetilde{\sfF}))$ is a cofinitely generated cotorsion $R$-module for $i=1,2$. Therefore, the Pontryagin dual of $\rmH^{i}(\widetilde{\sfF}/\sfF_{\infty}, A(\widetilde{\sfF}))$ is a pseudonull $R\lrbracket{\Gal(F_{\infty}/F)}$-module for $i = 1, 2$.
\end{theorem}

When $R=\ZZ_p$, the group $A(\sfL)$ is a cofinitely generated cotorsion $\ZZ_p$-module if and only if $A(\sfL)$ is finite, and if and only if $V(\sfL)=0$. Thus, Theorem \ref{thm:generalgreenberg} is a generalization of Theorem \ref{thm:greenberg}.

\begin{proof}
   By Lemma \ref{cofg or cotorsion}, $A(\sfL)=\rmH^0(\sfL, A)$ is a cofinitely generated $R$-module. We use the same idea as in the proof of Theorem \ref{thm:greenberg}. Similarly, we claim that the cotorsionness of $A(\sfL)$ for all finite extensions $\sfL/\sfF_\infty$ inside $\widetilde{\sfF}$ implies that there is an intermediate field $\sfF_{\circ}$ in $\widetilde{\sfF}/\sfF_{\infty}$ such that
   \begin{enumerate}[label = (\arabic*)]
    \item $A(\widetilde{\sfF})^{\Gal(\widetilde{\sfF}/\sfF_{\circ})}$ is cofinitely generated and cotorsion $R$-module, \label{item:R circ1}
    \item $Z := \Gal(\widetilde{\sfF}/\sfF_{\circ})$ is isomorphic to $\ZZ_p$. \label{item:R circ2}
\end{enumerate}
Granting the existence of such an intermediate field $\sfF_{\circ}$, we show using the Hochschild--Serre spectral sequence that the cohomology groups $\rmH^{i}(\widetilde{\sfF}/\sfF_{\circ}, A(\widetilde{\sfF}))$ are cofinitely generated and cotorsion for $i=1, 2$. Recall that a first-quadrant spectral sequence yields a seven-term exact sequence. Applying this to the Hochschild--Serre spectral sequence (see, for example, \cite[Remark 1.35 on page 72]{milneCFT}), we obtain
\[
\begin{aligned}
0 &\to \rmH^1(\sfF_{\circ}/\sfF_\infty, A(\widetilde{\sfF})^Z)
\xrightarrow{\inf}
\rmH^1(\widetilde{\sfF}/\sfF_\infty, A(\widetilde{\sfF}))
\xrightarrow{\res}
\rmH^1(\widetilde{\sfF}/\sfF_{\circ},A(\widetilde{\sfF}))^{\Gal(\sfF_{\circ}/\sfF_\infty)} \\
&\to \rmH^2(\sfF_{\circ}/\sfF_\infty, A(\widetilde{\sfF})^Z)
\xrightarrow{\inf}
\ker\!\left(
\rmH^2(\widetilde{\sfF}/\sfF_\infty, A(\widetilde{\sfF}))
\xrightarrow{\res}
\rmH^2(\widetilde{\sfF}/\sfF_{\circ},A(\widetilde{\sfF}))
\right) \\
&\to
\rmH^1(\sfF_{\circ}/\sfF_\infty,
\rmH^1(\widetilde{\sfF}/\sfF_{\circ}, A(\widetilde{\sfF})))
 \to
\rmH^3(\sfF_{\circ}/\sfF_\infty, A(\widetilde{\sfF})^Z)
\end{aligned}
\]
Since $Z=\Gal(\widetilde{\sfF}/\sfF_{\circ})\cong \ZZ_p$ has $p$-cohomological dimension 1, we have that $\rmH^2(\widetilde{\sfF}/\sfF_{\circ},A(\widetilde{\sfF}))=0$. Hence, 
\[
\begin{aligned}
0 &\to \rmH^1(\sfF_{\circ}/\sfF_\infty, A(\widetilde{\sfF})^Z)
\xrightarrow{\inf}
\rmH^1(\widetilde{\sfF}/\sfF_\infty, A(\widetilde{\sfF}))
\xrightarrow{\res}
\rmH^1(\widetilde{\sfF}/\sfF_{\circ},A(\widetilde{\sfF}))^{\Gal(\sfF_{\circ}/\sfF_\infty)} \\ 
&\to \rmH^2(\sfF_{\circ}/\sfF_\infty, A(\widetilde{\sfF})^Z)
\xrightarrow{\inf}
\rmH^2(\widetilde{\sfF}/\sfF_\infty, A(\widetilde{\sfF}))
\to
\rmH^1(\sfF_{\circ}/\sfF_\infty,
\rmH^1(\widetilde{\sfF}/\sfF_{\circ}, A(\widetilde{\sfF}))) \\
&\to
\rmH^3(\sfF_{\circ}/\sfF_\infty, A(\widetilde{\sfF})^Z).
\end{aligned}
\]
Our goal is to prove that the second and fifth terms are cotorsion $R$-modules.

By the defining property \ref{item:R circ1} of $\sfF_0$, the first term $\rmH^1(\sfF_{\circ}/\sfF_\infty, A(\widetilde{\sfF})^Z)$  and the fourth term 
$\rmH^2(\sfF_{\circ}/\sfF_\infty, A(\widetilde{\sfF})^Z) $ are cofinitely generated and cotorsion $R$-modules by \ref{cofg or cotorsion}. By Lemma \ref{dual of M}, $\rmH^1(\widetilde{\sfF}/\sfF_{\circ},A(\widetilde{\sfF}))$ is a cofinitely generated cotorsion $R$-module. Hence, by Lemma \ref{cofg or cotorsion}, both the third term $\rmH^1(\widetilde{\sfF}/\sfF_{\circ},A(\widetilde{\sfF}))^{\Gal(\sfF_{\circ}/\sfF_\infty)}$ and the sixth term $\rmH^1(\sfF_{\circ}/\sfF_\infty,
\rmH^1(\widetilde{\sfF}/\sfF_{\circ}, A(\widetilde{\sfF}))) $ are cofinitely generated and cotorsion $R$-modules. 
Hence, the second term $\rmH^1(\widetilde{\sfF}/\sfF_\infty, A(\widetilde{\sfF}))$ and the fifth term $\rmH^2(\widetilde{\sfF}/\sfF_\infty, A(\widetilde{\sfF})) $ are cofinitely generated cotorsion $R$-modules. The last statement of the theorem follows from Proposition \ref{prop:torsionpseudonull}. 

In the remaining part of this section, we prove the existence of an intermediate field $\sfF_{\circ}$ in $\widetilde{\sfF}/\sfF_{\infty}$ satisfying properties \ref{item:R circ1} and \ref{item:R circ2} above. This follows directly from Lemma \ref{lem:existFcirc} applied with $M = A(\widetilde{\sfF})$.
\end{proof}

\subsubsection{Existence of the intermediate field $\sfF_{\circ}$}
As promised, we prove Lemma \ref{lem:existFcirc} in this subsection, starting with a few technical lemmas.

\begin{lemma}\label{inv and coinv}
     Let $M$ be a cofinitely generated $R$-module with continuous $R$-linear $N\cong \ZZ_p^d$ action. Then $(M^N)^\vee\cong (M^\vee)_N$.
\end{lemma}
\begin{proof}
    This follows directly from the definition. Alternatively, \cite[Theorem 2.6.9]{MR1737196}, together with $\rmH^n(N,M)^\vee\cong \rmH_n(N,M^\vee)$ and $n=0$, yields the result.
\end{proof}

\begin{lemma} \label{nontrivial inv and coinv}
    Let $K$ be a field and let $V$ be a finite-dimensional $K$ vector space with continuous $K$-linear $N\cong\ZZ_p^d$ action. Then $V^N\neq 0$ if and only if $V_N\neq 0$. 
\end{lemma}
\begin{proof}
    Let $\gamma_1,\gamma_2,\cdots,\gamma_d$ be a set of topological generators of $N \cong\ZZ_p^d $. Write $T_i=\gamma_i-1$ for $1\leq i\leq d$. We see that
    \[
    V_N=V/(T_1V+T_2V+\cdots+T_dV)
    \]
    Consider dual vector spaces 
    \[
    V^*:=\Hom(V,K), \text{ and } (V_N)^*:=\Hom(V_N,K).
    \]
    For any $f\in \Hom(V_N,K)$, we have $f(T_iv)=0$, i.e., $(\gamma_if)(v)=f(v)$. Hence, $\gamma f=f$ for any $\gamma\in N$, so $(V^*)^N= (V_N)^*$. We obtain
    \[
    (V^*)^N\neq 0\Longleftrightarrow (V_N)^*\neq 0.
    \]
    We now compare $(V^*)^N$ and $V^N$. Let $\bar{K}$ be the algebraic closure of $K$ and let $\bar{V}:=V\otimes_K\bar{K}$. Then $(V^*)^N\otimes_K \bar{K}= (\bar{V}^*)^N $ and $ V^N\otimes_K \bar{K}=\bar{V}^N$. Hence,
    \[
    (V^*)^N\neq 0 \Longleftrightarrow (\bar{V}^*)^N \neq 0,
    \]
    \[
    V^N\neq 0\Longleftrightarrow \bar{V}^N\neq 0.
    \]
    View $\bar{V}$ and $\bar{V}^*$ as representations of the abelian group $N$ over the algebraically closed field $\bar{K}$. Notice that $\chi_{\bar{V}^*}(g)=\chi_{\bar{V}}(g^{-1})$ and each irreducible representation is one-dimensional. Therefore, the characters appearing in $\bar{V}$ are the inverses of those appearing in $\bar{V}^*$. Hence, the trivial character appears in $\bar{V}$ if and only if it appears in $\bar{V}^*$. Thus,
    \[
    (\bar{V}^*)^N\neq 0 \Longleftrightarrow \bar{V}^N\neq 0,
    \]
    which completes the proof.
\end{proof}

\begin{lemma}\label{commutative algebra}
    Let $M$ be a finitely generated $R$-module. Let $\Ass(R)$ be the set of associated primes of $R$. Then $M$ is a torsion $R$-module if and only if $ M\otimes_R Q(R/\frp)=0$ for any $ \frp\in \Ass(R)$.
\end{lemma}
\begin{proof}
    Set $S:=R \smallsetminus \{ \text{zero divisor of } R \}$.
Let $Q(R):=S^{-1}R$ be the total quotient ring of $R$. By the definition of a torsion module, an $R$-module $M$ is a torsion module if and only if $ M\otimes_R Q(R)=0$. Let $\Ass(R)$ be the set of associated primes of $R$. By \cite[Proposition 4.7]{atiyah1969introduction}, we have
\[
\{\text{zero divisor of } R\} = \bigcup_{\frp\in \Ass(R)}\frp.
\]
Hence, $R_\frp$ is a localization of the total quotient ring $Q(R)$ for $\frp\in \Ass(R)$. If $M\otimes_R Q(R)=0$, then $ M\otimes_R R_\frp=0$. Thus, $M\otimes_R Q(R/\frp)=M\otimes_R (R_\frp/\frp R_\frp)= 0$. This proves one direction.

We now prove the converse. Assume $ M\otimes_R Q(R)\neq 0$. By \cite[Proposition 3.8]{atiyah1969introduction}, there exists a prime ideal $\frq$ of $R$ such that $\frq\cap S=\emptyset$ and $ M\otimes_R R_\frq\neq 0$. Hence,
\[
\frq \subset \{\text{zero divisor of } R\} =\bigcup_{\frp\in \Ass(R)}\frp.
\]
By \cite[Proposition 1.11]{atiyah1969introduction}, there exists a prime ideal $\frp\in \Ass(R)$ such that $\frq\subset \frp$. Thus, $ M\otimes_R R_\frp\neq 0$. By Nakayama's lemma, we have $ M\otimes_R Q(R/\frp)=M\otimes_R (R_\frp/\frp R_\frp)\neq 0$. This proves the converse.
\end{proof}

We are now ready to prove Lemma \ref{lem:existFcirc}, and hence end the proof of Theorem \ref{thm:generalgreenberg}.

\begin{lemma} \label{lem:existFcirc}
    Let $M$ be a cofinitely generated $R$-module with a continuous $R$-linear $N\cong \ZZ_p^d$-action. If $M^\calN$ is a cofinitely generated cotorsion $R$-module for every finite-index subgroup $\calN\subset N$, then there exists a subgroup $Z\subset N$ with $Z\cong \ZZ_p$ such that $M^Z$ is a cofinitely generated cotorsion $R$-module.
\end{lemma}
\begin{proof}
We consider everything on its Pontryagin dual side. By Lemma \ref{inv and coinv}, we have that $(M^\vee)_\calN$ is a finitely generated and torsion $R$-module for any finite index subgroup $\calN\subset N$. 
Let $\Ass(R)$ be the set of associated primes of $R$. By Lemma \ref{commutative algebra}, for any $\frp\in \Ass(R)$, we have that $(M^\vee)_\calN\otimes_R Q(R/\frp)=0$ for any finite index subgroup $\calN\subset N$. 

View $M^\vee\otimes_R Q(R/\frp)$ as a representation of $N$ over the field $ K_\frp :=Q(R/\frp)$. Extend the representation linearly to its algebraic closure $\overline{K_\frp}$. Set $V_\frp:=M^\vee\otimes_R \overline{K_\frp}$. Then $(V_\frp)_{\calN}=0$ for any finite-index subgroup $\calN\subset N$. View $V_\frp$ as a finite-dimensional representation of the abelian group $N$ over the algebraically closed field $\overline{K_\frp}$. All of its irreducible representations over $\overline{K_\frp}$ are one-dimensional. The action of $N$ on the composition factors is given by the homomorphisms $\chi_{\frp,i}: N\to \overline{K_\frp}^* $ for $1\leq i\leq r_\frp$, where $r_\frp$ is the dimension of $V_\frp$ over $\overline{K_\frp}$. By assumption, the kernel of $\chi_{\frp,i}$ has infinite index in $N$. Indeed, otherwise the restriction $\chi_{\frp,i}|_\calN $ would be trivial for some finite-index subgroup $\calN\subset N$. Then the composition factors corresponding to $\chi_{\frp,i}$ would contribute nontrivially to $(V_\frp)^\calN$. By Lemma \ref{nontrivial inv and coinv}, this would imply that $(V_\frp)_\calN$ is nontrivial, contradicting the assumption that $(V_\frp)_\calN=0$. Hence, $\ker(\chi_{\frp,i})\otimes_{\ZZ_p}\QQ_p$ is a proper subspace of $N\otimes_{\ZZ_p}\QQ_p$. We can choose $Z\cong \ZZ_p$ so that $Z\otimes_{\ZZ_p}\QQ_p$ is not contained in any of the proper subspaces $\ker(\chi_{\frp,i})\otimes_{\ZZ_p}\QQ_p$ of $N\otimes_{\ZZ_p}\QQ_p$, which is certainly possible since the set $\Ass(R)$ is finite. Hence, $V_\frp^Z=0$ for all $\frp \in \Ass(R)$. By Lemma \ref{nontrivial inv and coinv}, for all $\frp\in \Ass(R)$, we have $(V_\frp)_Z=(M^\vee\otimes_R \overline{K_\frp})_Z=0$, which further implies that $(M^\vee)_Z\otimes_R Q(R/\frp) = 0$. By Lemma \ref{commutative algebra}, $(M^\vee)_Z$ is a finitely generated torsion $R$-module. By Lemma \ref{inv and coinv}, $M^Z$ is a cofinitely generated cotorsion $R$-module.
\end{proof}

\subsection{Greenberg's finiteness conditions in families}

In Theorems \ref{thm:greenberg} and \ref{thm:generalgreenberg}, we see that the following condition is essential:
\begin{condition}
    \item[(RF)] $A(\sfL)$ is a cotorsion $R$-module for all finite extensions $\sfL/\sfF_\infty$ inside $\widetilde{\sfF}$. \label{eq:rf}
\end{condition}
We shall refer to \dref{eq:rf} as the \emph{Greenberg finiteness condition}.\footnote{We use the notation (RF), rather than (GF), to avoid conflict with the abbreviation ``global finiteness condition'' used elsewhere in the article. The letter ``R'' may also be viewed as emphasizing that our arguments are carried out over more general coefficient rings, rather than only over $\ZZ_p$. The letter ``R'' may also be viewed as an allusion to Professor Ralph Greenberg.} More specifically, we impose the following \emph{global and local finiteness conditions}:
\begin{condition}
    \item[(GF)] $A(L)$ is a cotorsion $R$-module for all finite extensions $L/F_\infty$ inside $\widetilde{F}$. \label{eq:gf}
    \item[(LF)] $A(L)$ is a cotorsion $R$-module for all finite extensions $L/F_{\infty, v}$ inside $\widetilde{F}_{w_{\infty}}$ for all $w_{\infty} \mid w \mid v \mid p$. \label{eq:lf}
\end{condition}
Here $v$ is a place of $F_\infty$ above $p$,  $w_\infty$ is a place of $\widetilde{F}$ above $v$ and $w$ is a place of $L$ above $v$ and below $w_\infty$.

We note that in \dref{eq:gf} and \dref{eq:lf}, the finite extensions $L/F_{\infty}$ and $L_w/F_{\infty,v}$ have degree $p^{t}$ for some integer $t \geq 0$. Moreover, condition \dref{eq:lf} implies condition \dref{eq:gf}. Nevertheless, for clarity of exposition, we still state condition \dref{eq:gf} explicitly in our results, as condition \dref{eq:gf} is also of independent interest.

Greenberg's method in the classical setup (Theorem \ref{thm:greenberg}) deals with the case $R = \ZZ_p$, and there are various tools to verify Greenberg's finiteness condition, as we shall see in Sections \ref{sec:exampleclassgroup}, \ref{sec:exampleellipticcurves} and \ref{sec:crystalline}. It is natural to ask whether these finiteness conditions propagate in their $p$-adic families.

For a Noetherian ring $R$ and an integer $n \geq 1$, we say that it satisfies \emph{Serre's condition $(S_n)$} if \[\depth(R_\frp)\geq \min\{n,\Ht_{R}(\frp)\}.\]
For example, any reduced ring satisfies $(S_1)$ by \cite[\href{https://stacks.math.columbia.edu/tag/031R}{Tag 031R}]{stacks-project}. Integrally closed Noetherian commutative integral domains satisfy Serre's condition $(S_2)$ by \cite[\href{https://stacks.math.columbia.edu/tag/031S}{Tag 031S}]{stacks-project}. Cohen--Macaulay rings satisfy Serre's condition $(S_n)$ for all $n$ by \cite[\href{https://stacks.math.columbia.edu/tag/0342}{Tag 0342}]{stacks-project}.

\begin{theorem}\label{torsion irr}
    Let $M$ be a finitely generated $R$-module. Assume $R$ satisfies Serre's condition $(S_1)$. Assume that on each irreducible component of $\Spec(R)$, there exists a prime $\wp$ such that $M/\wp M$ is a finitely generated and torsion $R/\wp$-module. Then $M$ is a finitely generated and torsion $R$-module.
\end{theorem}

\begin{proof}
    To show that $M$ is a torsion $R$-module, by Lemma \ref{commutative algebra}, it suffices to show that for any associated prime $\frp$ of $R$, we have $M \otimes_{R} Q(R/\frp) = 0$. 
    
    It follows from Serre's condition $(S_1)$ that $R$ has no embedded primes (see \cite[\href{https://stacks.math.columbia.edu/tag/031Q}{Lemma 031Q}]{stacks-project}), and hence any such associated prime $\frp$ is a minimal prime ideal of $R$, corresponding one-to-one to an irreducible component of $\Spec(R)$. We fix such an irreducible component $V(\frp)$ of $\Spec(R)$. The assumption then gives a prime ideal $\wp \in V(\frp)$ such that $M/\wp M$ is a torsion $R/\wp$-module. 
 
    Since $R/\wp$ is an integral domain, we have
    \[
    M/\wp M \otimes_{R/\wp} Q(R/\wp) = 0,
    \]
    and then
    \begin{align*}
        M_{\wp}/\wp M_{\wp} \simeq (M/\wp M) \otimes_{R/\wp} R_{\wp}/\wp R_{\wp} \simeq (M/\wp M) \otimes_{R/\wp} Q(R/\wp) = 0.    
    \end{align*}

    It then follows from Nakayama's lemma that $M_{\wp} = 0$. Since $\wp$ contains $\frp$, this further implies $M_{\frp} = 0$, and hence $M \otimes_{R} Q(R/\frp) = 0$, since
    \begin{align*}
        M \otimes_{R} Q(R/\frp) &\simeq M \otimes_{R} R_{\frp}/\frp R_{\frp} \\ &\simeq (M \otimes_{R} R_{\frp}) \otimes_{R_{\frp}} R_{\frp}/\frp \\ &\simeq M_{\frp} \otimes_{R_{\frp}} R_{\frp}/\frp = 0,    
    \end{align*}
    as desired.
\end{proof}

\begin{theorem} \label{thm:bigimai}
    Assume that $R$ satisfies Serre's condition $(S_1)$, and that on each irreducible component of $\Spec(R)$, there exists a prime $\wp$ such that $A^{(\wp)} := A[\wp]$ satisfies the \dref{eq:rf} condition. Then $A$ satisfies the \dref{eq:rf} condition.
\end{theorem}
\begin{proof}
    Condition \dref{eq:rf} means that $A[\wp](\sfL)=A[\wp]^{\Gal_\sfL}$ is a cofinitely generated cotorsion $R/\wp$-module for all finite extensions $\sfL/\sfF_\infty$ inside $\tilde{\sfF}$. In other words, the Pontryagin dual 
    \[ 
    A[\wp](\sfL)^\vee=(A^\vee)_{\Gal_\sfL}/\wp(A^\vee)_{\Gal_\sfL}
    \] 
    is a finitely generated torsion $R/\wp$-module. By Theorem \ref{torsion irr}, $(A^\vee)_{\Gal_\sfL}$ is a finitely generated torsion $R$-module. Hence, $ A^{\Gal_\sfL}$ is a cofinitely generated cotorsion $R$-module. In other words, $A$ satisfies condition \dref{eq:rf}.
\end{proof}

\section{Vertical control theorem}
\label{sec:greenbergfinite}

We are now ready to present the vertical control theorem, i.e. Theorem \ref{thm:verticalcontrol}. Let $F$ be a number field as before and consider a tower of multiple $\ZZ_p$-extensions $\tilF / F_{\infty} / F$, with 
\[
\Gal(F_{\infty}/F) = \ZZ_p, \, \Gal(\tilF/F) = \ZZ_p^{n+1} \text{ for } n \geq 1. 
\]
We remark that the compositum of all $\ZZ_p$-extensions of $F$ has Galois group isomorphic to $\ZZ_p^{r_2(F)+1+\delta(F)}$ over $F$, where $r_{2}(F)$ is the number of pairs of complex embeddings of $F$, and $\delta(F)$ is the \emph{Leopoldt defect}, which is conjectured to be zero. Hence $n \leq r_2(F)+\delta(F)$.

\begin{theorem} \label{thm:verticalcontrol}
Assume that $F_{\infty}/F$ is cyclotomic-like. Assume that conditions \dref{eq:gf} and \dref{eq:lf} hold. Then the kernel and the cokernel of the Pontryagin dual $\res_{\tilF/F_{\infty}}^{\vee}$ of the restriction map
\[
\res_{\tilF/F_{\infty}}: \Sel(A/F_{\infty}) \rightarrow \Sel(A/\tilF)^{\Gal(\tilF/F_{\infty})},
\]
are both finitely generated torsion $R$-modules. In particular, as a morphism of $R\lrbracket{\Gal(F_{\infty}/F)}$-modules, $\res_{\tilF/F_{\infty}}^{\vee}$ is a pseudo-isomorphism.
\end{theorem}

\begin{proof}
Consider the following commutative diagram with exact rows:
\begin{equation} \label{tikz:ctrlKtil}
\begin{tikzcd}
	0 & {\Sel(A/F_{\infty})} & {\rmH^{1}(F_{S}/F_{\infty}, A)} & {\prod_{v \in S_{F_{\infty}}} \rmH^{1}(F_{\infty,v}, A)} \\
	0 & {\Sel(A/\tilF)^{N}} & {\rmH^{1}(F_{S}/\tilF, A)^{N}} & {\left(\prod_{w \in S_{\tilF}} \rmH^{1}(\tilF_{w}, A) \right)^{N}}
	\arrow[from=1-1, to=1-2]
	\arrow[from=1-2, to=1-3]
	\arrow["{\res_{\tilF/F_{\infty}}}"', from=1-2, to=2-2]
	\arrow[from=1-3, to=1-4]
	\arrow["{\alpha_{\tilF/F_{\infty}}}", from=1-3, to=2-3]
	\arrow["{\beta_{\tilF/F_{\infty}}}", from=1-4, to=2-4]
	\arrow[from=2-1, to=2-2]
	\arrow[from=2-2, to=2-3]
	\arrow[from=2-3, to=2-4]
\end{tikzcd},
\end{equation}
where $N := \Gal(\tilF/F_{\infty})$. By the snake lemma, it suffices to study $\ker(\alpha_{\tilF/F_{\infty}})$, $\coker(\alpha_{\tilF/F_{\infty}})$, and $\ker(\beta_{\tilF/F_{\infty}})$.

(1) Applying Theorem \ref{thm:generalgreenberg} to the case $\sfF = F$, we see that the Pontryagin dual of
\[
\ker(\alpha_{\tilF/F_{\infty}}) =  \ker \left( \rmH^{1}(F_{\infty}, A) \rightarrow \rmH^{1}(\tilF, A)^{N} \right) \simeq \rmH^{1}(\tilF/F_{\infty}, A(\tilF))
\]
is a pseudonull $R\lrbracket{\Gal(F_{\infty}/F)}$-module, assuming \dref{eq:gf}. Here the isomorphism $\simeq$ follows from the inflation-restriction exact sequence. 

(2) By the inflation-restriction short exact sequence, we see that
\begin{align*}
  \coker(\alpha_{\tilF/F_{\infty}}) &\simeq \ker(\rmH^{2}(\tilF/F_{\infty}, A(\tilF)) \rightarrow \rmH^{2}(F_{S}/F_{\infty}, A) )\\ 
  &\hookrightarrow \rmH^{2}(\tilF/F_{\infty}, A(\tilF)),  
\end{align*}
and hence the Pontryagin dual of $\coker(\alpha_{\tilF/F_{\infty}})$ is a quotient of the Pontryagin dual of $\rmH^{2}(\tilF/F_{\infty}, A(\tilF))$. It therefore follows from Theorem \ref{thm:generalgreenberg} that the Pontryagin dual of $\rmH^{2}(\tilF/F_{\infty}, A(\tilF))$ is a pseudonull $R\lrbracket{\Gal(F_{\infty}/F)}$-module under the assumption \dref{eq:gf}. Therefore, the Pontryagin dual of $\coker(\alpha_{\tilF/F_{\infty}})$ is a pseudonull $R\lrbracket{\Gal(F_{\infty}/F)}$-module.

(3) Since $F_{\infty}/F$ is cyclotomic-like, the set $S_{F_{\infty}}$ is finite and the restriction map $\beta_{\tilF/F_{\infty}}$ can be decomposed into a \emph{finite} product of maps
\[
\beta_{\tilF/F_{\infty}} = \prod_{v \in S_{F_{\infty}}} \beta_{\tilF/F_{\infty}, v}
\]
with each factor being the natural restriction map
\[
\beta_{\tilF/F_{\infty}, v}: \rmH^{1}(F_{\infty,v}, A) \rightarrow \prod_{w \in S_{\tilF}, \, w \mid v} \rmH^{1}(\tilF_{w}, A)^{N_w}, \quad N_{w} := \Gal(\tilF_{w}/F_{\infty,v}).
\]
Note that although there may be infinitely many places $w$ of $\tilF$ lying above $v$, the kernel of each restriction map
\[
\beta_{\tilF/F_{\infty}, w}: \rmH^{1}(F_{\infty,v}, A) \rightarrow \rmH^{1}(\tilF_{w},A)^{N_w}
\]
is independent of the choice of the place $w$ lying over $p$.

As in the analysis of the global restriction map $\alpha_{\tilF/F_{\infty}}$, we have
\begin{align*}
\ker(\beta_{\tilF/F_{\infty}, v}) &= \ker \left( \rmH^{1}(F_{\infty,v}, A) \rightarrow \rmH^{1}(\tilF_{w}, A)^{N_w} \right) \\ &\simeq \rmH^{1}(\tilF_{w}/F_{\infty,v}, A(\tilF_{w}))
\end{align*}
for any fixed $w$ lying over $v$, where the last isomorphism $\simeq$ follows from the inflation-restriction exact sequence. We separate the discussion according to whether $v$ lies above $p$.

(3.a) If $v \in S_{F_{\infty}}$ lies above $p$, then the Pontryagin dual of the cohomology group $\rmH^{1}(\tilF_{w}/F_{\infty,v}, A(\tilF_{w}))$ is a pseudonull $R\lrbracket{\Gal(F_{\infty}/F)}$-module by Theorem \ref{thm:generalgreenberg} applied to the case $\sfF = F_{v_0}$ where $v_0$ is a place of $F$ lying below $v$, assuming \dref{eq:lf}. 

(3.b) Suppose $v \in S_{F_{\infty}}$ does not lie above $p$. Let $v_0$ be the place of $F$ lying below $v$. By local class field theory, let $F_{v_0}^{\natural}$ denote the compositum of all $\ZZ_p$-extensions of $F_{v_0}$. Then $\Gal(F_{v_0}^{\natural}/F_{v_0}) \simeq \ZZ_p$ since $v$ does not lie above $p$. Since we have assumed that $F_{\infty}/F$ is cyclotomic-like, the decomposition group $\Gal(F_{\infty, v}/F_{v_0})$ is a nontrivial subgroup of $\ZZ_p$, and hence $\Gal(F_{\infty, v}/F_{v_0}) \simeq \ZZ_p$. Therefore, $\tilF_{w} = F_{\infty, v}$ for any $v$ not lying over $p$. It follows that $\ker(\beta_{\tilF/F_{\infty}, v})$ is zero for all $v \in S_{F_{\infty}}$ not lying above $p$.

Combining cases (3.a) and (3.b), we see that  the Pontryagin dual of $\ker(\beta_{\tilF/F_{\infty}})$ is a pseudonull $R\lrbracket{\Gal(F_{\infty}/F)}$-module. The snake lemma applied to the diagram \eqref{tikz:ctrlKtil} shows that the restriction map $\res_{\tilF/F_{\infty}}$ is a pseudo-isomorphism, as desired.
\end{proof}

\begin{remark}
We remark that our vertical control theorem (Theorem \ref{thm:verticalcontrol}) includes the stronger assertion that the kernel and cokernel of the Pontryagin dual map $\res_{\tilF/F_{\infty}}^{\vee}$ are finitely generated torsion $R$-modules. This implies their pseudonullity as $R\lrbracket{\Gal(F_{\infty}/F)}$-modules by Proposition \ref{prop:torsionpseudonull}, but not vice versa. Indeed, a pseudonull module over an Iwasawa algebra $R\lrbracket{T_{1}, \ldots, T_{n}}$ need not be finitely generated over $R$. For example, take $R=\ZZ_p\lrbracket{X}$ and $\Lambda=R\lrbracket{T}$. The $\Lambda$-module $M := \Lambda/(p,X) = \FF_p \lrbracket{T}$ is pseudonull over $\Lambda$, but $M$ is not finitely generated over $R$. \footnote{Indeed, viewing $M$ as a $\Lambda$-module, the annihilator of $M$ is
the prime ideal $(p,X)$ of $\Lambda$, which is of height $2$, which implies that $M$ is a pseudonull $\Lambda$-module. However, $M$ is not finitely generated as an $R$-module. In fact, both $p$ and $X$ act trivially on $M$, so the $R$-action
factors through $R/(p,X) \simeq \FF_p$.
If $M$ were finitely generated over $R$, it would therefore be a
finite-dimensional $\FF_p$-vector space. This is impossible because
the elements $1,T,T^2,T^3,\ldots$ are linearly independent over $\FF_p$ in
$\FF_p\lrbracket{T}$. Hence $M$ is pseudonull over $\Lambda$ but is not finitely generated over $R$.}
\end{remark}

\section{Coinvariant criterion for pseudonullity and Conjecture \ref{myconj:B}} \label{sec:descent}

In this subsection, we collect several descent results that allow one to detect torsionness and pseudonullity by passing to suitable quotients. Recall that in the previous sections, $R$ is a profinite $\ZZ_p$-algebra and $R$ is commutative and Noetherian.

\subsection{Coinvariant criterion for pseudonullity}

The following lemma was stated as \cite[Lemma 7]{fujii2017on}.
\begin{lemma}
For a positive integer $d > 1$, let $G$ be a pro-$p$ group such that $G \simeq \ZZ_p^d$ as topological groups. Let $H$ be a subgroup of $G$ such that $G/H \simeq \ZZ_p^{d-1}$. Let $M$ be a finitely generated $\ZZ_p\lrbracket{G}$-module. If the $H$-coinvariant module $M_H$ of $M$ is pseudonull over $\ZZ_p\lrbracket{G/H}$, then $M$ is pseudonull over $\ZZ_p\lrbracket{G}$.
\end{lemma}

In this section, we generalize the lemma from the coefficient ring $\ZZ_p$ to a general coefficient ring $R$, assumed only to be Noetherian and Cohen--Macaulay; see Theorem \ref{thm:descentpseudonull}. 

\begin{theorem}[A coinvariant criterion for pseudonullity]\label{catenary}
    Let $\Lambda$ be a commutative Noetherian profinite ring and assume that $\Lambda$ is Cohen--Macaulay. Let $T$ be a regular element in $\Lambda$ and assume $T$ is in the Jacobson radical. Let $M$ be a finitely generated $\Lambda$-module. If $M/TM$ is a pseudonull $\Lambda/T\Lambda$-module, then $M$ is a pseudonull $\Lambda$-module.
\end{theorem}

\begin{proof}
    We start with the following preliminary observation. Let $\pi: \Lambda \rightarrow \Lambda/T\Lambda$ be the natural quotient map, and let $\frp$ be a prime ideal in $\Lambda$ containing $T$. Then $\pi(\frp) = \frp/(T)$ is a prime ideal in $\Lambda/T\Lambda$.

    Notice that if 
    \[
    \overline{\frp}_0 \subset \overline{\frp}_1 \subset \cdots \subset \overline{\frp} := \pi(\frp)
    \]
    is a chain of prime ideals in $\Lambda/T\Lambda$, then
    \[
    \pi^{-1}(\overline{\frp}_0) \subset \pi^{-1}(\overline{\frp}_1) \subset \cdots \subset \pi^{-1}(\overline{\frp}) = \frp
    \]
    is a chain of prime ideals in $\Lambda$. Hence $\Ht_{\Lambda/T\Lambda}(\pi(\frp)) \leq \Ht_{\Lambda}(\frp)$. Moreover, since $T \in \pi^{-1}(\overline{\frp}_0)$ and $T$ is a \emph{regular} element in $\Lambda$, we have $\Ht(\pi^{-1}(\overline{\frp}_0)) \geq 1$. Therefore, we have the stronger inequality
    \begin{equation} \label{eq:height}
        \Ht_{\Lambda/T\Lambda}(\pi(\frp)) + 1 \leq \Ht_{\Lambda}(\frp).
    \end{equation}

    Now, to show that $M$ is a pseudonull $\Lambda$-module, it suffices to show that $M_{\frp} = 0$ for any prime ideal $\frp$ in $\Lambda$ with height $\Ht_{\Lambda}(\frp) \leq 1$.

    \underline{Case (1): $T \in \frp$}. Then $\Ht_{\Lambda/T\Lambda}(\overline{\frp}) \leq \Ht_{\Lambda}(\frp) \leq 1$. Since $M/TM$ is a pseudonull $\Lambda/T\Lambda$-module, we have $(M/TM)_{\overline{\frp}} = M_{\frp}/T M_{\frp} = 0$. Since $\Lambda_{\frp}$ is a local ring and $T$ lies in the maximal ideal of it (as $T \in \frp$), by Nakayama's lemma, we have $M_{\frp} = 0$. 

    \underline{Case (2): $T \not\in \frp$}. Since $T$ lies in the Jacobson radical of $\Lambda$, the ideal $(T)+\frp$ is a proper ideal of $\Lambda$ by \cite[Proposition 1.9]{atiyah1969introduction}. 
    Let $\frq$ be a minimal prime ideal of $\Lambda$ containing $(T)+\frp$. Then $\frq/\frp$ is a minimal prime ideal containing the ideal $((T)+\frp)/\frp$ of $\Lambda/\frp$. Since $((T)+\frp)/\frp$ is precisely the principal ideal of $\Lambda/\frp$ generated by $T$, and $T$ is neither a zero-divisor nor a unit of $\Lambda/\frp$, Krull's principal ideal theorem (see \cite[Corollary 11.17]{atiyah1969introduction}) gives $\Ht_{\Lambda/\frp}(\frq/\frp) \leq 1$. Since $\Lambda$ is a Cohen--Macaulay ring, it follows from \cite[\href{https://stacks.math.columbia.edu/tag/00NA}{Tag 00NA}]{stacks-project} that
    \begin{equation} \label{eq:CM_condition}
    \Ht_{\Lambda}(\frq) = \Ht_{\Lambda}(\frp) + \Ht_{\Lambda/\frp}(\frq/\frp),        
    \end{equation}
    which implies $\Ht_{\Lambda}(\frq) \leq 2$. Let $\overline{\frq}$ be the image of $\frq$ in $\Lambda/T\Lambda$, then the observation \eqref{eq:height} applied to $\frq$ implies that $\Ht(\overline{\frq}) \leq 1$, since $T \in \frq$. By the same arguments in Case (1) above, we see that $M_{\frq} = 0$. Since $\frp$ is contained in $\frq$, this further implies that $M_{\frp} = 0$, as desired.
\end{proof}

\begin{theorem}[A coinvariant criterion for pseudonullity]
\label{thm:descentpseudonull}
Let $R$ be a commutative Noetherian profinite Cohen--Macaulay ring.
Let $G$ be a pro-$p$ group such that $G \simeq \ZZ_p^{d}$ for some integer $d\geq1$. Let $H$ be a closed subgroup of $G$ such that $G/H \simeq \ZZ_p$. Let $M$ be a finitely generated $R\lrbracket{G}$-module. If the $H$-coinvariant module $M_H$ is pseudonull over $R\lrbracket{G/H}$, then $M$ is pseudonull over $R\lrbracket{G}$.
\end{theorem}

\begin{proof}
By choosing a suitable basis $\{ \gamma_1,\cdots,\gamma_d \}$ of $G \simeq \ZZ_p^{d}$ for $d \geq 1$, we can write 
\[
R\lrbracket{G}=\Lambda = R \lrbracket{T_1, \ldots, T_d}, \, R\lrbracket{G/H}=\Lambda/(T_1,\cdots,T_{d-1})
\]
and 
\[
M_H=M/(T_1,\cdots,T_{d-1})M.
\]
Then the theorem follows if we inductively apply Theorem \ref{catenary} by taking $\Lambda/(T_1,\cdots,T_i)$ to be $\Lambda$ and $T_{i+1}$ to be $T$ in the setting of Theorem \ref{catenary}.

To apply Theorem \ref{catenary}, we need to check:
\begin{enumerate}
    \item The ring $\Lambda/(T_1,\cdots,T_i)$ is Cohen--Macaulay and 
    \item in the ring $\Lambda/(T_1,\cdots,T_i)$, the element $T_{i+1}$ is regular and lies in its Jacobson radical.
\end{enumerate}

For (1), notice that $\Lambda/(T_1,\cdots,T_i)\cong R\lrbracket{T_{i+1},\cdots, T_d}$. Since $R$ is Cohen--Macaulay, by Theorem 23.5 and the remark following it in \cite{Matsumura_1987}, the ring $R\lrbracket{T_{i+1},\cdots, T_d}$ is Cohen--Macaulay.
For (2), note that in the ring $\Lambda/(T_1,\cdots,T_i)$, the element $T_{i+1}$ is a nonzero divisor and hence, by definition, regular. Since $T_{i+1}$ is topologically nilpotent, $(1-xT_{i+1})$ is a unit for any $x\in \Lambda/(T_1,\cdots,T_i)$. It follows from \cite[Proposition 1.9]{atiyah1969introduction} that $T_{i+1}$ is in the Jacobson radical. This proves the theorem.
\end{proof}

\subsection{Main theorems on Conjectures \ref{myconj:Bminus} and \ref{myconj:Bplus}}

We start by recording the following result on the finite generation of dual fine Selmer groups $Y_{S}(A/\tilF)$ over $R\lrbracket{\Gal(\tilF/F)}$.

\begin{proposition} \label{prop:fgselmer}
Let $R$ be a commutative complete Noetherian local ring with residual characteristic $p$. Let $\tilF/F$ be a $p$-adic Lie extension of $F$ contained in $F_{S}$. Then
\begin{enumerate}[label = \rm (\arabic*)]
    \item The ring $R\lrbracket{\Gal(\tilF/F)}$ is Noetherian,
    \item $\rmH^{2}_{S}(\tilF/F, T^{\ast})$ is finitely generated over $R\lrbracket{\Gal(\tilF/F)}$, 
    \item $Y_{S}(A/\tilF)$ is finitely generated over $R\lrbracket{\Gal(\tilF/F)}$.
\end{enumerate}
\end{proposition}
\begin{proof}
    Item (1) follows from \cite[Proposition 3.0.1]{lim2013}, and item (2) follows from \cite[Proposition 4.1.3]{lim2013}. Item (3) is then a consequence of (1), (2), and the Poitou--Tate exact sequence \eqref{eq:poitoutate}.
\end{proof}

From now on, we assume that $R$ is a commutative complete Noetherian local ring with residual characteristic $p$. This is necessary to apply Proposition \ref{prop:fgselmer} to justify the application of Theorem \ref{thm:descentpseudonull} in the proof of the following theorem.

\begin{theorem}[{Conjecture \ref{myconj:Bminus} and \ref{myconj:Bplus}}] \label{thm:maintheorem}
Let $F$ be a number field. Suppose $F_{\infty}/F$ is a cyclotomic-like $\ZZ_p$-extension, and let $\tilF/F$ be a $\ZZ_p^{d}$-extension with $d>1$ such that $\tilF$ contains $F_{\infty}$. If conditions \dref{eq:gf} and \dref{eq:lf} hold for $F_{\infty}/F$, then we have the following results.
\begin{enumerate}[label = \rm (\arabic*)]
    \item \emph{(Conjecture \ref{myconj:Bminus})} If $Y(A/F_{\infty})$ is a finitely generated $R$-module, then $Y(A/\tilF)$ is a finitely generated $R\lrbracket{\Gal(\tilF/F_{\infty})}$-module.
    \item \emph{(Conjecture \ref{myconj:Bplus})} Assume further that $R$ is a Cohen--Macaulay ring. If $Y(A/F_{\infty})$ is a pseudonull $R\lrbracket{\Gal(F_{\infty}/F)}$-module, then $Y(A/\tilF)$ is a pseudonull  $R\lrbracket{\Gal(\tilF/F)}$-module.
\end{enumerate}
\end{theorem}

\begin{proof}
Under the conditions \dref{eq:gf} and \dref{eq:lf}, we can apply the vertical control theorem (Theorem \ref{thm:verticalcontrol}) to see that there is a pseudo-isomorphism of $R\lrbracket{\Gal(F_{\infty}/F)}$-modules
\[
\res_{\tilF/F_{\infty}}^{\vee}: Y(A/\tilF)_{\Gal(\tilF/F_{\infty})} \rightarrow Y(A/F_{\infty})
\]
with kernel and cokernel being finitely generated torsion $R$-modules.

For (1), under the assumption that $Y(A/F_{\infty})$ is a finitely generated $R$-module, we see that the coinvariant module $Y(A/\tilF)_{\Gal(\tilF/F_{\infty})}$ is a finitely generated $R$-module. It follows from the topological Nakayama lemma (see \cite[Corollary on page 226]{balister1997}) that $Y(A/\tilF)$ is a finitely generated $R\lrbracket{\Gal(\tilF/F_{\infty})}$-module.

For (2), since $Y(A/F_{\infty})$ is a pseudonull $R\lrbracket{\Gal(F_{\infty}/F)}$-module, we see that $Y(A/\tilF)_{\Gal(\tilF/F_{\infty})}$ is a pseudonull $R\lrbracket{\Gal(F_{\infty}/F)}$-module. We apply Theorem \ref{thm:descentpseudonull} to see that $Y(A/\tilF)$ is a pseudonull $R\lrbracket{\Gal(\tilF/F)}$-module.
\end{proof}

\section{Applications} \label{sec:applications}

In this section, we give several applications of our main theorem, namely Theorem \ref{thm:maintheorem}. In each example, we first verify Greenberg's finiteness condition and then apply Theorem \ref{thm:maintheorem}. Before doing so, however, we give a counterexample showing that the vertical control theorem may fail if Greenberg's finiteness condition \dref{eq:rf} is violated.

We retain the notation from the previous sections. Let $F$ be a number field, let $F_{\infty}/F$ be a cyclotomic-like $\ZZ_p$-extension, and let $\tilF/F$ be a $\ZZ_p^{d}$-extension containing $F_{\infty}/F$.

\subsection{Counterexample: $A = \QQ_p/\ZZ_p$} \label{sec:counter} We first consider the case $A = \QQ_p/\ZZ_p$. In this case, it is easy to see that 
\[
A(F_{\cyc}) = \rmH^{0}(F_{\cyc}, \QQ_p/\ZZ_p) = \QQ_p/\ZZ_p,
\]
so Greenberg's finiteness condition \dref{eq:gf} fails. In this case, we do not expect the vertical control theorem (i.e. Theorem \ref{thm:verticalcontrol}) to hold.

Let $F$ be an imaginary quadratic field over $\QQ$, and let $\tilF$ be the compositum of all $\ZZ_p$-extensions of $F$. Then it is known that $\Gal(\tilF/F) \simeq \ZZ_p^{2}$, with $\tilF/F$ being the compositum of the cyclotomic $\ZZ_p$-extension $F_{\cyc}/F$ and the anticyclotomic $\ZZ_p$-extension $F_{\ac}/F$ of $F$. In this setup, we have the following proposition.

\begin{proposition} \label{prop:counter}
    If $p$ splits in $F$ and $p$ does not divide the class number of $F$, then there is a short exact sequence of $\ZZ_p\lrbracket{\Gal(F_{\cyc}/F)}$-modules
    \[
    0 \rightarrow \QQ_p/\ZZ_p \rightarrow \Sel(F_{\cyc}, \QQ_p/\ZZ_p) \xrightarrow{\res_{\tilF/F_{\cyc}}} \Sel(\tilF, \QQ_p/\ZZ_p)^{\Gal(\tilF/F_{\cyc})} \rightarrow 0.
    \]
    In particular, the Pontryagin dual of the restriction map $\res_{\tilF/F_{\cyc}}$ has zero kernel and non-pseudonull cokernel.
\end{proposition}

For example, if $F = \QQ(\sqrt{-1})$ and $p = 5$, then $p$ splits in $F$ and the class number of $F$ is $1$.

We note that the Pontryagin dual of the fine Selmer group $ \Sel(F_{\cyc}, \QQ_p/\ZZ_p)$ is, in fact, the Iwasawa module $\varprojlim_{F_n}\Cl_S(F_n)$, where $\Cl_S(F_n)$ is the $S$-class group of $F_n$ and $S$ is the set of primes above $p$.
\begin{proof}
    Returning to the proof of Theorem \ref{thm:verticalcontrol}, we see that there is an exact sequence involving $\res := \res_{\tilF/F_{\cyc}}$:
    \begin{align} \label{eq:counter}
        0 & \rightarrow \ker(\res) \rightarrow \rmH^{1}(\Gamma, \QQ_p/\ZZ_p) \rightarrow \prod_{v \in S_{F_{\cyc}}, v \mid p} \rmH^{1}(\Gamma_v, \QQ_p/\ZZ_p) \cap \im(\gamma_{F_{\cyc}}) \notag \\ & \rightarrow \coker(\res) \rightarrow 0, 
    \end{align}
    where $\Gamma := \Gal(\tilF/F_{\cyc})$ and $\Gamma_{v} := \Gal(\tilF_{v}/F_{\cyc, v})$, with $\gamma_{\cyc}$ the global-to-local map
    \[
    \gamma_{\cyc}: \rmH^{1}(F_{\cyc}, \QQ_p/\ZZ_p) \rightarrow \prod_{v \in S_{F_{\cyc}}} \rmH^{1}(F_{\cyc, v}, \QQ_p/\ZZ_p).
    \]
    Under the assumption that $p$ does not divide the class number of $F$, it follows from \cite[Remark 3.1]{kundu2024cotorsion} that primes above $p$ in $F$ are totally ramified in $F_{\ac}/F$. Since $p$ splits in $F$ into $p \calO_{F} = \frp \overline{\frp}$, and there is a unique totally ramified $\ZZ_p$-extension of $\QQ_p$, namely the field $\QQ_{p,\cyc}$ inside $\QQ_p(\mu_{p^{\infty}})$, we see that
    \[
    F_{\frp} = \QQ_p, \quad F_{\overline{\frp}} = \QQ_p, \qquad \tilF_{\frp} = \QQ_{p,\cyc}, \quad \tilF_{\overline{\frp}} = \QQ_{p,\cyc}.
    \]
    Therefore, for any place $v$ of $F_{\cyc}$ lying over $p$, we have $\Gamma_{v} = 0$, and hence it follows from \eqref{eq:counter} that
    \[
    \ker(\res) = \rmH^{1}(\Gamma, \QQ_p/\ZZ_p) = \QQ_p/\ZZ_p, \text{ and } \coker(\res) = 0.
    \]
    The proposition is therefore proved.
\end{proof}

However, this does \emph{not} necessarily imply that the dual Selmer group $Y(\tilF, \QQ_p/\ZZ_p)$ is not a pseudonull $\ZZ_p\lrbracket{\Gal(\tilF/F)}$-module, since the converse of Theorem \ref{thm:descentpseudonull} may not hold. For example, consider the two-variable Iwasawa algebra $\Lambda = \ZZ_p\lrbracket{T_1, T_2}$ and the $\Lambda$-module $M := \Lambda/(p, T_2)$. Then $M$ is a pseudonull $\Lambda$-module since $\Ht_{\Lambda}((p, T_2)) = 2$. However, the coinvariant module $M/T_{2}M \simeq \ZZ_p\lrbracket{T_1}/p$ is not pseudonull over $\ZZ_p\lrbracket{T_1}$.

If we replace the cyclotomic $\ZZ_p$-extension $F_{\cyc}/F$ with the $\ZZ_p$-extension $F^{(1)}/F$ in Example \ref{eg:fieldK1}, we have not found any counterexample for which $\res_{\tilF/F^{(1)}}$ is not a pseudonull isomorphism. 

\subsection{Example: $A = \QQ_p/\ZZ_p(1)$} \label{sec:exampleclassgroup}

Section \ref{sec:counter} shows that, when $A = \QQ_p/\ZZ_p$, Greenberg's finiteness condition \dref{eq:lf} fails and the vertical control theorem does not hold. However, after applying a Tate twist to $A$, the situation is different.

In this example, we keep the following assumption:
\begin{condition}
    \item[(Spl)]  $p$ splits completely in $F$. \label{item:spl}
\end{condition}

We consider the vertical control theorem for $A = \QQ_p/\ZZ_p(1) = \mu_{p^{\infty}}$ when $F_{\infty}$ is the cyclotomic $\ZZ_p$-extension of $F$. Let $L$ be a finite extension of $F_{\cyc}$. Since $p$ splits completely in $F$, we have that $F$ does not contain $\mu_p(\barF)$. Therefore, $[F(\mu_{p^{\infty}}):F_{\cyc}] = p-1$. For any finite extension $L/F_{\cyc}$ such that $[L:F_{\cyc}] = p^{t}$ for $t \geq 0$, it follows that $A(L) = \mu_{p^{\infty}}(L) = 0$, and hence \dref{eq:gf} holds. For the local conditions, consider any place $v$ of $F$ above $p$. Then $v$ is totally ramified in $F_{\cyc}$. Since $p$ splits completely in $F$, we have $F_{v} = \QQ_p$. Recall that by local class field theory for $\QQ_p$, there are exactly two linearly disjoint $\ZZ_p$-extensions of $\QQ_p$: the unramified $\ZZ_p$-extension $\QQ_{p}^{\circ}$ and the cyclotomic $\ZZ_p$-extension $\QQ_{p, \cyc}$, which is a subfield of $\QQ_p(\mu_{p^{\infty}})$. For any wild finite extension $L_{w}/F_{\cyc,v}$, we have $\mu_{p^{\infty}}(L_{w}) = \{1\}$ since $[\QQ_p(\mu_p):\QQ_p] = p-1$. Therefore, \dref{eq:lf} holds for $v$ as well, and the finiteness conditions \dref{eq:gf} and \dref{eq:lf} are verified for $A = \QQ_p/\ZZ_p(1)$.

Therefore, the finiteness conditions \dref{eq:gf} and \dref{eq:lf} are verified, as desired. The following theorem is therefore a consequence of Theorems \ref{thm:verticalcontrol} and \ref{thm:maintheorem}.

\begin{theorem}[Main theorems for Tate's motive]
\label{thm:apptate}
With the notation as above, assume that \dref{item:spl} holds. Then we have the following results.
\begin{enumerate}[label = \rm (\arabic*)]
    \item \emph{(Vertical control theorem)} The Pontryagin dual of the restriction map
    \[
    \res_{\tilF/F_{\cyc}}: \Sel(\mu_{p^{\infty}}/F_{\cyc}) \rightarrow \Sel(\mu_{p^{\infty}}/\tilF)^{\Gal(\tilF/F_{\cyc})}
    \]
    is a pseudo-isomorphism of $\ZZ_p\lrbracket{\Gal(F_{\cyc}/F)}$-modules with kernel and cokernel being finitely generated torsion $\ZZ_p$-modules.
    \item \emph{(Conjecture \ref{myconj:Bminus})} If $Y(\mu_{p^{\infty}}/F_{\cyc})$ is a finitely generated $\ZZ_p$-module, then $Y(\mu_{p^{\infty}}/\tilF)$ is a finitely generated $\ZZ_p\lrbracket{\Gal(\tilF/F_{\cyc})}$-module.
    \item \emph{(Conjecture \ref{myconj:Bplus})} If $Y(\mu_{p^{\infty}}/F_{\cyc})$ is a pseudonull $\ZZ_p\lrbracket{\Gal(F_{\cyc}/F)}$-module, then $Y(\mu_{p^{\infty}}/\tilF)$ is a pseudonull  $\ZZ_p\lrbracket{\Gal(\tilF/F)}$-module.
\end{enumerate}
\end{theorem}

\subsection{Example: Elliptic curves} \label{sec:exampleellipticcurves}
Let $E$ be an elliptic curve over a number field $F$ and let $T := T_{p}(E)$ be its Tate module, carrying an action of $\Gal_{F}$. In this case, $A$ is simply $E[p^{\infty}]$. We shall write $\Sel_{S}(E/\calL) := \Sel_{S}(E[p^{\infty}]/\calL)$ for simplicity. Let $v$ be a place of $F$ above $p$. It follows from the semistable reduction theorem (see \cite[Proposition 5.4]{silverman2009arithmetic}) that $E$ has either potentially good reduction or potentially multiplicative reduction at $v$. We consider these two cases separately in the following subsections.

\subsubsection{Potential good reduction case}
When $E$ has potentially good reduction at the places of $F$ above $p$, condition \dref{eq:lf} holds for the cyclotomic $\ZZ_p$-extension $F_{\infty}^{\cyc}/F$ of $F$ by Hideo Imai's classical result \cite{imai1975}. We record his result in the following proposition.
\begin{proposition} \label{prop:imai}
    Let $E$ be an elliptic curve over a number field $F$ with potentially good reduction at a place $v$ of $F$ above $p$. Then, for any number field $F$, any place $v$ of $F$ above $p$, and any finite extension $L/F_{\cyc, v}^{\infty}$, the group $E(L)[p^{\infty}]$ is finite.
\end{proposition}

We now apply Theorems \ref{thm:verticalcontrol} and \ref{thm:maintheorem} to the elliptic curve $E$. The following result combines them with Proposition \ref{prop:imai}.

\begin{theorem}[Main theorems for elliptic curves: potential good reduction case] \label{thm:appec}
    Let $E$ be an elliptic curve over $F$ with potentially good reduction at all places of $F$ above $p$. Then we have the following results.
\begin{enumerate}[label = \rm (\arabic*)]
    \item \emph{(Vertical control theorem)} The Pontryagin dual of the restriction map
    \[
    \res_{\tilF/F_{\cyc}}: \Sel(E/F_{\cyc}) \rightarrow \Sel(E/\tilF)^{\Gal(\tilF/F_{\cyc})}
    \]
    is a pseudo-isomorphism of $\ZZ_p\lrbracket{\Gal(F_{\cyc}/F)}$-modules, with kernel and cokernel being finitely generated torsion $\ZZ_p$-modules.
    \item \emph{(Conjecture \ref{myconj:Bminus})} If $Y(E/F_{\cyc})$ is a finitely generated $\ZZ_p$-module, then $Y(E/\tilF)$ is a finitely generated $\ZZ_p\lrbracket{\Gal(\tilF/F_{\cyc})}$-module.
    \item \emph{(Conjecture \ref{myconj:Bplus})} If $Y(E/F_{\cyc})$ is a pseudonull $\ZZ_p\lrbracket{\Gal(F_{\cyc}/F)}$-module, then $Y(E/\tilF)$ is a pseudonull  $\ZZ_p\lrbracket{\Gal(\tilF/F)}$-module.
\end{enumerate}
\end{theorem}

We add an interlude to justify the condition in (2) of Theorem \ref{thm:appec} in some cases. A drawback is that the pseudonullity of $Y(E/F_{\cyc})$ does not always hold. As remarked in \cite[page 826]{coates2005fine}, for $E/\QQ$ such that $E(\QQ)$ has $\ZZ$-rank greater than $1$, the dual fine Selmer module $Y(E/\QQ_{\cyc})$ has positive $\ZZ_p$-rank. Nevertheless, we do have some sufficient conditions to guarantee the pseudonullity of $Y(E/F_{\cyc})$.

In the following proposition and its proof, we denote $\Sel_{\cl}(E/\calL)$ the classical $p$-primary Selmer group of the elliptic curve $E$ over any algebraic extension $\calL$ of $\QQ$ contained in $\QQ_{S}$. For $\calL = F_{\cyc}$ for number fields $F$, we denote $X(E/F_{\cyc})$ the Pontryagin dual of $\Sel_{\cl}(E/F_{\cyc})$, which is a finitely generated $\ZZ_p\lrbracket{\Gal(F_{\cyc}/F)}$-module.

\begin{proposition} \label{coro:elliptic2}
Let $E / \QQ$ be an elliptic curve with good reduction at $p$. Suppose that
\begin{enumerate}[label = \rm (G-\arabic*)]
    \item $\Sel_{\cl}(E/F)$ is finite and $p \nmid \abs{\Sel_{\cl}(E/F)}$, \label{item:G1} \footnote{By the fundamental exact sequence \[
    0 \rightarrow E(F) \otimes_{\ZZ} \QQ_p/\ZZ_p \rightarrow
\Sel_{\cl}(E/F) \rightarrow \Sha(E/F)[p^{\infty}] \rightarrow 0,   \]
This condition implies that $E(F)$ has rank zero (and hence $E(\QQ)$ has rank zero as well) and $\Sha(E/F)[p^{\infty}] = 0$.}
    \item $p \nmid \Tam_{v}(E/F)$ for any finite place $v$ of $F$. \label{item:G2}
\end{enumerate}
If $E$ has good ordinary reduction at $p$, we further assume that
\begin{enumerate}[label = \rm (O-\arabic*)]
    \item $E(F)[p] = 0$, \label{item:O1}
    \item $p \nmid \abs{\tilE(\kappa_v)}$ for any place $v$ of $F$ above $p$ (these primes are called \emph{non-anomalous}). \label{item:O2}
\end{enumerate}
If $E$ has good supersingular reduction at $p$, we further assume that
\begin{enumerate}[label = \rm (S-\arabic*)]
    \item $p \geq 5$ and $p$ is unramified in $F$. \label{item:S1}
\end{enumerate}
Then $Y(E/F_{\cyc})$ is a pseudonull $\ZZ_p\lrbracket{\Gal(F_{\cyc}/F)}$-module.
\end{proposition}

\begin{proof}
Since $\Gal(F_{\cyc}/F) \simeq \ZZ_p$, the pseudonullity of $Y(E/F_{\cyc})$ is equivalent to its finiteness.

We first deal with the case where $E$ has good ordinary reduction at $p$. It then follows from Mazur's control theorem (see \cite[Corollary 4.9]{greenberg2002introduction}) that the $p$-primary Selmer group $\Sel_{\cl}(E/F_{\cyc})$ is $\ZZ_p\lrbracket{\Gal(F_{\cyc}/F)}$-cotorsion. Let $f_{E}(T)$ be the characteristic power series of $X(E/F_{\cyc})$, then $f_{E}(0) \neq 0$ since $\Sel_{\cl}(E/F)$ is finite under condition \ref{item:G1}. The refined control theorem of Greenberg, proved as \cite[Theorem 4.1]{greenberg1999iwasawa}, implies that
\[
v_p(f_{E}(0)) = v_p \left( \dfrac{\prod_{v} \Tam_{v}(E/F) \prod_{v \mid p} \abs{\tilE(\kappa_v)}^2 \abs{\Sel_{\cl}(E/F)}}{\abs{E(F)[p]}^{2}} \right).
\]
By assumptions \ref{item:G1}, \ref{item:G2}, \ref{item:O1} and \ref{item:O2}, $f_E(0)$ is a $p$-adic unit. This implies that $X(E/F_{\cyc})$, and hence $Y(E/F_{\cyc})$, are finite, as desired. \footnote{In fact, as we have assumed \ref{item:O1}, by \cite[Proposition 4.14]{greenberg1999iwasawa}, $X(E/F_{\cyc})$ has no nontrivial pseudonull $\ZZ_p\lrbracket{\Gal(F_{\cyc}/F)}$-submodules. Therefore, $X(E/F_{\cyc})$, and hence $Y(E/F_{\cyc})$, are actually trivial under our assumptions.} 

Next we deal with the case where $E$ has good supersingular reduction at $p$. In this case, we have a notion of $\pm$-Selmer groups when $p>3$, denoted by $\Sel^{\circ}(E/F_{\cyc})$ for $\circ \in \{ +, - \}$. Since $\Sel_{\cl}(E/F)$ is finite by assumption \ref{item:G1}, and under \ref{item:S1}, it is known that $\Sel^{\circ}(E/F_{\cyc})$ is $\ZZ_p\lrbracket{\Gal(F_{\cyc}/F)}$-cotorsion \footnote{See \cite[first lime of the proof of Corollary 3.15]{kim2013}.}. In this case, we have characteristic power series $f_{E}^{\circ}(T)$ of the Pontryagin dual $X^{\circ}(E/F_{\cyc})$ of $\Sel^{\circ}(E/F_{\cyc})$. When $\Sel_{\cl}(E/F)$ is finite and \ref{item:S1} holds, it follows from \cite[Theorem 1.2]{kim2013} that
\[
v_p (f_{E}^{\circ}(0)) = v_p \left( \prod_{v} \Tam_{v}(E/F) \abs{\Sel_{\cl}(E/F)} \right).
\]
Under assumptions \ref{item:G1} and \ref{item:G2}, $f_{E}^{\circ}(0)$ is a $p$-adic unit, and hence $X^{\circ}(E/F_{\cyc})$ is finite. It follows from definitions that $\Sel(E/F_{\cyc})$ is a subgroup of $\Sel^{\circ}(E/F_{\cyc})$, and hence $Y(E/F_{\cyc})$ is finite, as desired.
\end{proof}

\begin{remark}
The proof of \ref{coro:elliptic2} follows closely from \cite[Section 3]{kundu2024}. Moreover, in \cite[Theorem 3.1]{kundu2024}, it is proved that if $E/\QQ$ has good ordinary reduction at $p$ and, additionally, $E$ has CM by an order in an imaginary quadratic field $\calK$ such that $F$ contains $\calK$, then the set of prime numbers $p$ satisfying \ref{item:G1}, \ref{item:G2}, \ref{item:O1} and \ref{item:O2} has density at least $1/[F^c:\QQ]$, where $F^c$ is the Galois closure of $F$. When $E$ has good supersingular reduction at $p$, and under condition \ref{item:G1}, there are only finitely many prime numbers $p$ violating \ref{item:G2} and \ref{item:S1}.
\end{remark}

\subsubsection{Split multiplicative case} \label{sec:app_potential_multiplicative}

Let $E$ be an elliptic curve over a number field $F$ with multiplicative reduction at a place $v$ of $F$ above $p$. We start by assuming that $E$ has \emph{split} multiplicative reduction at $v$. Then, by \cite[Theorem V.5.3]{silverman1994advanced}, one has Tate's parameterization
\[
\barF_v^{\times}/\lrangle{q_E} \xrightarrow{\sim} E(\barF_v),
\]
where $\lrangle{q_E}$ is the infinite cyclic subgroup of $F_v^{\times}$ generated by the \emph{Tate period} $q_E \in F_{v}^{\times}$ for $E$, which has positive $v$-adic valuation. We have $\mu_p^{\infty} \cap \lrangle{q_E} = 1$, so the Tate parametrization (which is $\Gal_{F_v}$-equivariant) induces a natural inclusion $\mu_{p^{\infty}} \hookrightarrow E(\barF_v)$. One then obtains an exact sequence
\begin{equation} \label{eq:multiplicative}
    0 \rightarrow \mu_{p^{\infty}} \rightarrow E[p^{\infty}] \rightarrow Q \rightarrow 0,
\end{equation}
of $\Gal_{F_v}$-modules, with $Q$ defined as the quotient $E[p^{\infty}]/\mu_{p^{\infty}}$, with trivial $\Gal_{F_v}$-action. There is a rigorous identification
\[
\iota_E: Q \xrightarrow{\sim} \QQ_p/\ZZ_p,
\]
which sends the class of the point in $E[p^{\infty}]$ represented by $q_E^{m/p^n}$ to $m/p^n$. See Appendix \ref{app:tate} for additional details.

Let $L$ be any algebraic extension of $F_v$. To study $E(L)[p^{\infty}]$, we consider the following exact sequence
\begin{equation} \label{eq:multises}
0 \rightarrow \rmH^0(L, \mu_{p^{\infty}}) \rightarrow \rmH^{0}(L, E[p^{\infty}]) \rightarrow \rmH^{0}(L, Q) \xrightarrow{\delta_E(L)} \rmH^{1}(L, \mu_{p^{\infty}})
\end{equation}
obtained by taking $\Gal_{L}$-cohomology in \eqref{eq:multiplicative}. The $p^{\infty}$-torsion points of $E(L)$ thus come from two parts: the roots-of-unity part $\mu_{p^{\infty}}(L)$ and the part arising from $\ker(\delta_E)(L)$. We now describe the kernel of the connecting homomorphism
\[
\delta_E(L): \rmH^{0}(L, Q) \rightarrow \rmH^{1}(L, \mu_{p^{\infty}})
\]
precisely in the following proposition.

\begin{proposition} \label{prop:kerneldelta}
Suppose that $E$ has split multiplicative reduction at a place $v$ of $F$ above $p$, with Tate period $q_{E} \in \barF_{v}^{\times}$. Then
\begin{equation}\label{eq:describedelta}
    \ker(\delta_E)(L) \simeq \{m/p^n \in \QQ_p/\ZZ_p: q_E^m\in L^{\times {p^n}}\}.
\end{equation}
\end{proposition}

Let $\calL \subseteq \calL^{\prime}$ be any two algebraic extensions of $\QQ_p$. We say that an element $x \in \calL$ is \emph{infinitely $p$-divisible} in $\calL^{\prime}$ if $x \in (\calL^{\prime \times})^{p^{n}}$ for every $n \geq 0$; that is, for every $n \geq 0$, there exists $x_{n} \in \calL^{\prime \times}$ such that $x_{n}^{p^{n}} = x$.

\begin{remark} \label{rem:infinitely_p_divisible}
    We remark that if in addition $\calL/\QQ_p$ is a finite extension, then any $x \in \calL$ with $v_{p}(x) > 0$ is not an infinitely $p$-divisible element in $\calL$. Otherwise this would imply that $x \in \calL^{\times p^m}$ for every integer $m \geq 0$. If this were true, for arbitrarily large $m$ we could write
    \[
    x = x_m^{p^m} \quad \text{ for some } x_m \in \calL^{\times}.
    \]
    Taking the $p$-adic valuation on both sides yields
    \[
    v_p(x) = v_p(x_m^{p^m}) = p^m v_p(x_m).
    \]
    Since $v_p(x)=v>0$ is a fixed integer and $v_p(x_m)$ is an integer, this forces $p^m$ to divide $v$ for every $m \geq 0$, a contradiction.
\end{remark}

From this perspective, Proposition \ref{prop:kerneldelta} tells us that the size of $\ker(\delta_E)(L)$ depends on the $p$-divisibility of $q_E$ inside $L^{\times}$. We remark that Proposition \ref{prop:kerneldelta} (more accurately, the commutative diagram \eqref{eq:diagramkummer}) already appears in \cite[Section 4]{hida2009invariants}. Nevertheless, we supply an explicit proof here.

\begin{proof}
Following the description of \eqref{eq:multiplicative} above, we have the identification
\[
\iota_{E}: \rmH^0(L,Q) \xrightarrow{\sim} \QQ_p/\ZZ_p.
\]
The key to this proposition is to verify that the following diagram commutes:
\begin{equation} \label{eq:diagramkummer}
\begin{tikzcd}
    \rmH^{0}(L, Q) & & \rmH^{1}(L, \mu_{p^{\infty}}) \\ 
    \QQ_p/\ZZ_p & & \varinjlim_{n} L^{\times}/L^{\times p^n} 
    \arrow[from=1-1, to=1-3, "\delta_E(L)"]
    \arrow[from=1-1, to=2-1, "\iota_{E}"', "\simeq"]
    \arrow[to=1-3, from=2-3, "\Kum"', "\simeq"]
    \arrow[from=2-1, to=2-3, "\delta_{E}^{\prime}(L)"']
\end{tikzcd}, 
\end{equation}
where the lower horizontal map is defined by
\[
\delta^{\prime}_E(L): \QQ_p/\ZZ_p \rightarrow \varinjlim_{n} L^{\times}/L^{\times p^n}, \quad m/p^{n} \mapsto q_E^m \cdot L^{\times p^n},
\]
and the right vertical map $\Kum$ is defined via Kummer theory. Indeed, classical Kummer theory gives the map
    \[
    \Kum: \varinjlim_{n} L^{\times}/L^{\times p^n}  \xrightarrow{\sim} \rmH^{1}(L, \mu_{p^{\infty}}).
    \]
For any $[q] := \{q \cdot L^{\times p^{n}}\}_n \in \varinjlim_{n} L^{\times}/L^{\times p^n}$, the image of $[q]$ under the Kummer map $\Kum$ is the cocycle
    \begin{equation} \label{eq:cocyclekum}
    \left[ c_{[q]}: \sigma \in \Gal_{L} \longmapsto c_{[q]}(\sigma) = \left( \frac{\sigma(q^{1/p^n})}{q^{1/p^n}} \in \mu_{p^n} \right)_{n \geq 0} \right] \in \rmH^{1}(L, \mu_{p^{\infty}}).
    \end{equation}
    Applying this to $[q_{E}^{m}] = \{q_{E}^{m} \cdot L^{\times p^{n}}\}_n$, we obtain the cocycle $\delta_{E}(L)([q_{E}^{m/p^{n}}])$ for $[q_{E}^{m/p^{n}}] \in Q$. The commutativity of \eqref{eq:diagramkummer} then follows. Finally, the diagram \eqref{eq:diagramkummer} shows that
\[
    \ker(\delta_E)(L) \simeq \ker(\delta_{E}^{\prime})(L) = \{m/p^n \in \QQ_p/\ZZ_p: q_{E}^{m} \in L^{\times {p^n}}\},
\]
as desired.
\end{proof}

\begin{proposition} \label{prop:imaimult}
    Let $E$ be an elliptic curve over $F$ with split multiplicative reduction at the place $v$ of $F$ above $p$. Let $L$ be any algebraic extension of $F_v$. If $\mu_{p^{\infty}}(L)$ is finite and $q_{E}$ is not infinitely $p$-divisible in $L^{\times}$, then $E(L)[p^{\infty}]$ is finite.
\end{proposition}

\begin{proof}
Assuming that $E$ has split multiplicative reduction at $v$, \eqref{eq:multises} gives the following short exact sequence:
\[
0 \rightarrow \mu_{p^{\infty}}(L) \rightarrow E[p^{\infty}](L) \rightarrow \ker(\delta_{E})(L) \rightarrow 0.
\]
Hence $E(L)[p^{\infty}]$ is finite if and only if both $\mu_{p^{\infty}}(L)$ and $\ker(\delta_{E})(L)$ are finite. By Proposition \ref{prop:kerneldelta}, the latter amounts to considering the $p^{\infty}$-divisibility of $q_{E}$ inside $L$. More precisely,
\begin{itemize}
    \item Suppose that $q_{E}$ is not infinitely $p$-divisible in $L$; that is, there exists some $N$ such that $q_{E} \not\in L^{\times p^{N}}$.
   Hence,  the connecting homomorphism $\delta$ is not a zero map. We have $\ker(\delta_E)(L)$ is finite.
    \item Suppose that $q_E$ is infinitely $p$-divisible. Then the connecting homomorphism $\delta$ is the zero map, with $\ker(\delta_E)(L) = \QQ_p/\ZZ_p$.
\end{itemize}
The conclusion then follows.
\end{proof}

Proposition \ref{prop:imaimult}, together with the preceding discussion, shows that the finiteness of $E(L)[p^{\infty}]$ is controlled by both the $p$-power roots-of-unity part $\mu_{p^{\infty}}(L)$ and the infinite $p$-divisibility of the Tate period $q_{E}$. The former has been discussed in Section \ref{sec:exampleclassgroup}. In the remainder of this section, we study the infinite $p$-divisibility of $q_{E}$. The main result is the following lemma.

\begin{lemma}\label{lem:q_E_divisible}
Let $K$ be a finite extension of $\QQ_p$, and let $L/K$ be a
possibly infinite abelian Galois extension. If $q\in K^\times$
satisfies $v_K(q)\neq 0$, then $q$ is not infinitely
$p$-divisible in $L^\times$.
\end{lemma}

\begin{proof}
Suppose, to the contrary, that $q\in L^{\times p^n} $ for every $n \geq 1$. One may choose a compatible system
\[
\underline{q} := (q_{n})_{n \geq 0} \in (L^\times)^{\ZZ},
\quad \text{ with } q_{n+1}^p=q_n, \, q_0 = q.
\]

Recall that we denote $\Gal_{K} := \Gal(\barK/K)$ and $\Gal_{K}^{\ab} := \Gal(K^{\ab}/K)$. Since $L\subseteq K^{\ab}$ by assumption, every $q_n$ belongs to $K^{\ab}$. The compatible system defines the Kummer
cocycle
\[
\left[ c_q \colon \Gal_K\longrightarrow \ZZ_p(1); \quad c_q(\sigma)  =
\left(\dfrac{\sigma(q_n)}{q_n}\right)_{n\geq 1} \right] \in \rmH^{1}(\Gal_{K}, \ZZ_p(1)).
\]
This cocycle vanishes on $\Gal(\barK/K^{\ab})$. Moreover, $\Gal(\barK/K^{\ab})$ acts trivially on $\ZZ_p(1)$ since the cyclotomic extension $K(\mu_{p^\infty})/K$ is abelian. Consequently, $c_q$ is inflated from a class in $\rmH^1(\Gal_K^{\ab},\ZZ_p(1))$.

We claim that \underline{$\rmH^1(\Gal_K^{\ab},\ZZ_p(1))$ is a finite group}. Let $\chi_{\cyc}\colon \Gal_K^{\ab}\longrightarrow\ZZ_p^\times$ be the cyclotomic character and put $B=\ker(\chi_{\cyc})$. We choose $\tau\in \Gal_K^{\ab}$ such that $\chi_{\cyc}(\tau)\neq 1$. If $c: \Gal_K^{\ab}\to\ZZ_p(1)$ is a continuous cocycle and $b\in B$, then, since $G_K^{\ab}$ is abelian, $c(\tau b)=c(b\tau)$. Using the cocycle relation on both sides gives
\[
c(\tau)+\chi_{\cyc}(\tau)c(b)= c(b)+c(\tau),
\]
and hence $(\chi_{\cyc}(\tau)-1)c(b)=0$.
Since $\ZZ_p(1)$ is torsion-free and $\chi_{\cyc}(\tau)-1\neq0$, it follows that $c(b) = 0$. This shows that every cocycle in $\rmH^{1}(\Gal_{K}^{\ab}, \ZZ_p(1))$ vanishes on $B$ and factors through $\chi_{\cyc}(G_K^{\ab})\subseteq\ZZ_p^\times$. The group $\chi_{\cyc}(\Gal_K^{\ab})$ contains an open procyclic subgroup $C\simeq\ZZ_p$. Let $\gamma$ be a topological generator of $C$. Then
\[
\rmH^1(C,\ZZ_p(1)) \simeq \ZZ_p/\left(\chi_{\cyc}(\gamma)-1\right)\ZZ_p,
\]
which is finite. Since $C$ has finite index in $\chi_{\cyc}(\Gal_K^{\ab})$, it follows that $\rmH^{1}(\Gal_{K}^{\ab}, \ZZ_p(1))$ is finite.

On the other hand, Kummer theory gives
\[
\rmH^1(\Gal_K,\ZZ_p(1))
 \simeq
\varprojlim_n K^\times/K^{\times p^n},
\]
and under the valuation map, the Kummer class of $q$ maps to
\[
\left(v_K(q)\bmod p^n \right)_{n\geq1} \in \varprojlim_n\ZZ/p^n\ZZ \simeq\ZZ_p.
\]
Since $v_K(q) \neq 0$, this element has infinite order. Therefore, the Kummer class of $q$ has infinite order in $\rmH^{1}(\Gal_{K}^{\ab}, \ZZ_p(1))$. This contradicts the fact that it is inflated from the finite group
$\rmH^1(\Gal_K^{\ab},\ZZ_p(1))$. Hence $q$ cannot be infinitely $p$-divisible in $L^{\times}$.
\end{proof}

Applying Theorems \ref{thm:verticalcontrol} and \ref{thm:maintheorem}, together with Section \ref{sec:exampleclassgroup} controlling $\mu_{p^{\infty}}(L)$ and Lemma \ref{lem:q_E_divisible} controlling the $p$-divisibility of $q_{E}$ in $L$, we obtain the following result.

\begin{theorem}[Main theorems for elliptic curves: split multiplicative reduction case] \label{thm:appecmultcyc}
    Let $E$ be an elliptic curve over $F$ with split multiplicative reduction at all places $v$ of $F$ above $p$. Assuming \dref{item:spl}, we have the following results.
\begin{enumerate}[label = \rm (\arabic*)]
\item \emph{(Vertical control theorem)} The Pontryagin dual of the restriction map
    \[
    \res_{\tilF/F_{\cyc}}: \Sel(E/F_{\cyc}) \rightarrow \Sel(E/\tilF)^{\Gal(\tilF/F_{\cyc})}
    \]
    is a pseudo-isomorphism of $\ZZ_p\lrbracket{\Gal(F_{\cyc}/F)}$-modules, with kernel and cokernel being finitely generated torsion $\ZZ_p$-modules.
    \item \emph{(Conjecture \ref{myconj:Bminus})} If $Y(E/F_{\cyc})$ is a finitely generated $\ZZ_p$-module, then $Y(E/\tilF)$ is a finitely generated $\ZZ_p\lrbracket{\Gal(\tilF/F_{\cyc})}$-module.
    \item \emph{(Conjecture \ref{myconj:Bplus})} If $Y(E/F_{\cyc})$ is a pseudonull $\ZZ_p\lrbracket{\Gal(F_{\cyc}/F)}$-module, then $Y(E/\tilF)$ is a pseudonull  $\ZZ_p\lrbracket{\Gal(\tilF/F)}$-module.
\end{enumerate}
\end{theorem}

\subsubsection{Nonsplit potential multiplicative reduction} \label{sec:nonsplit}
Let $K=F_v$, and suppose that $E/F_{v}$ has potential multiplicative yet nonsplit multiplicative reduction. In this case, there exists a quadratic character $\eta: \Gal_K \rightarrow \{\pm1\}$ such that the quadratic twist $E^{\eta}$ of $E$ by $\eta$ has split multiplicative reduction at $v$. Moreover, if $E/F_v$ has nonsplit multiplicative reduction, then $\eta$ is the (unique) unramified character of $K$, and when $E/F_v$ has additive potential multiplicative reduction, then $\eta$ is a ramified quadratic character. For these facts, see Theorem \ref{thm:twist-elliptic}.

Following \eqref{eq:multiplicative} (and under the identification $\iota_{E^{\eta}}$), we have the following short exact sequence of $\Gal_{K}$-modules (for a proof, see Theorem \ref{prop:twisted_tate_sequence}):
\begin{equation}\label{eq:nonsplit-multiplicative}
0\rightarrow \mu_{p^\infty}\otimes\eta \rightarrow E[p^\infty] \rightarrow \QQ_p/\ZZ_p \otimes\eta \rightarrow 0.
\end{equation}
Here $\Gal_{K}$ acts on $\mu_{p^\infty}\otimes\eta$ through the character $\chi_{\cyc}\eta$, while it acts on $\QQ_p/\ZZ_p \otimes \eta$ through $\eta$.

The following proposition is parallel to Proposition \ref{prop:imaimult} in the split multiplicative case.

\begin{proposition}\label{prop:nonsplit-mult}
Suppose that $p$ is odd and that $E/K$ has nonsplit multiplicative
reduction. Let $L$ be an algebraic extension of $K$. Assume that
\[
\eta|_{\Gal_L}\neq 1
\quad\text{and}\quad
(\chi_{\cyc} \eta)|_{\Gal_L}\neq 1.
\]
Then $E(L)[p^\infty]$ is finite.
\end{proposition}

\begin{proof}
Taking $\Gal_{L}$-invariants in \eqref{eq:nonsplit-multiplicative} gives
the left exact sequence
\begin{equation} \label{eq:nonsplit_H0}
0 \rightarrow \rmH^{0}(\Gal_{L}, \mu_{p^{\infty} \otimes \eta}) \rightarrow \rmH^{0}(\Gal_{L}, E[p^{\infty}]) \rightarrow \rmH^{0}(\Gal_{L}, \QQ_p/\ZZ_p \otimes \eta).
\end{equation}
To make our life easier, we write
\[
A=\mu_{p^\infty}\otimes\eta,\quad B=E[p^\infty],\quad C=\QQ_p/\ZZ_p\otimes\eta.
\]

We first show that $\rmH^{0}(\Gal_{L}, C) = 0$. \footnote{This is the difference between the split and the nonsplit cases. In the split case, if we go back to the proof of Proposition \ref{prop:imaimult}, we see that $\rmH^{0}(\Gal_{L}, \QQ_p/\ZZ_p)$ might not be zero and hence we have to study $\ker(\delta_E)(L)$ in detail. However here, it is plausible to show that $\rmH^{0}(\Gal_{L}, \QQ_p/\ZZ_p \otimes \eta) = 0$.} Since $\eta|_{\Gal_L}\neq 1$, there exists $\sigma\in \Gal_L$ such that $\eta(\sigma)=-1$. If $x\in C^{\Gal_L}$, then $x$ is fixed by $\sigma$, but the definition of the twisted action also gives $\sigma(x)=\eta(\sigma)x=-x$, and hence $2x = 0$. The group $C$ is $p$-primary, and $p$ is odd, so this implies $x = 0$. Hence $\rmH^{0}(\Gal_{L}, C) = 0$.

We then show that $\rmH^{0}(\Gal_{L}, A)$ is finite. We write $\alpha:=\chi_{\cyc}\eta$. By assumption, $\alpha|_{\Gal_L}\neq 1$. Choose $\tau \in \Gal_{L}$ such that $a:=\alpha(\tau)\in \ZZ_p^{\times}$ is not equal to $1$. If $y \in A^{\Gal_L}$, then $a y=\tau(y)=y$ and hence $(a-1)y=0$. Write $a-1=p^m u$ with $m \geq 0$ (note that $m$ is independent of $y$) and $u \in \ZZ_p^{\times}$. Since multiplication
by $u$ is an automorphism of $\QQ_p/\ZZ_p$, the relation $(a-1)y = 0$ implies $p^{m}y = 0$. Thus $A^{\Gal_L}\subseteq A[p^m]$. The group $A[p^m]$ is finite of order $p^m$. Hence $A^{\Gal_L}$ is finite.

Combining the two preceding paragraphs with \eqref{eq:nonsplit_H0}, we see that $E(L)[p^{\infty}]$ is finite.
\end{proof}

The following lemma is parallel to Lemma \ref{lem:q_E_divisible} in the split multiplicative case.

\begin{lemma} \label{lem:nontrivial_criterion}
    Suppose that $K = \QQ_p$ and $L$ is a Galois pro-$p$ extension of the cyclotomic $\ZZ_p$-extension $\QQ_{p,\cyc}$ of $\QQ_p$. This includes the case where $L/\QQ_{p,\cyc}$ is a finite Galois extension of $p$-power degree. Suppose $\eta$ is a quadratic character of $\QQ_p$. When $\eta$ is ramified, we further assume that $p \geq 5$. Then we have \[
\eta|_{\Gal_L}\neq 1
\quad\text{and}\quad
(\chi_{\cyc} \eta)|_{\Gal_L}\neq 1.
\]
\end{lemma}
\begin{proof}

\underline{(1) Treat the character $\eta$-part.} 
First, the character $\eta$ cuts out a quadratic extension $K_{\eta}/\QQ_p$. On the other hand, $\QQ_{p,\cyc}/\QQ_p$ is $\ZZ_p$-extension. Since $p$ is an odd prime, we have
\[
K_{\eta} \cap \QQ_{p,\cyc} = \QQ_p, \quad \text{and} \quad [K_\eta \QQ_{p,\cyc}: \QQ_{p,\cyc}]=2.
\]
If $\eta|_{\Gal_{L}}$ were trivial, then $L$ would contain $K_{\eta} \QQ_{p, \cyc}$. This is impossible because $L/\QQ_{p,\cyc}$ is pro-$p$ whereas $[K_\eta \QQ_{p,\cyc}: \QQ_{p,\cyc}]=2$ and $p$ is odd. Therefore, $\eta|_{\Gal_{L}} \neq 1$.

\underline{(2) Treat the character $\chi_{\cyc}\eta$.}
It remains to show that $(\chi_{\cyc} \eta)|_{\Gal_{L}} \neq 1$, or equivalently, $\chi_{\cyc}|_{\Gal_{L}} \neq \eta|_{\Gal_{L}}$. We split the proof into two steps:

(2.a) Put $H := \Gal_{\QQ_{p,\cyc}}$. The restriction $\chi_{\cyc}|_{H}$ has image $\mu_{p-1} \subset \ZZ_p^{\times}$, and it cuts out the extension $\QQ_{p,\cyc}(\mu_p)/\QQ_{p,\cyc}$, which is a totally ramified extension of degree $p-1$. On the other hand, $\eta|_{H}$ cuts out the field $K_{\eta}\QQ_{p,\cyc}/\QQ_{p,\cyc}$, which is a quadratic extension.
\begin{itemize}
    \item When $\eta$ is an unramified quadratic character, then the quadratic extension $K_{\eta}\QQ_{p,\cyc}/\QQ_{p,\cyc}$ is unramified, and hence its intersection with the ramified extension $\QQ_{p,\cyc}(\mu_p)/\QQ_{p,\cyc}$  is $\QQ_{p,\cyc}$. It follows that $\chi_{\cyc}|_{H} \neq \eta|_{H}$.
    \item When $\eta$ is a ramified character, since we have assumed $p \geq 5$, by comparing the degree, we also obtain that the intersection of $K_{\eta}\QQ_{p,\cyc}/\QQ_{p,\cyc}$ with $\QQ_{p,\cyc}(\mu_p)/\QQ_{p,\cyc}$  is exactly $\QQ_{p,\cyc}$ , which gives the same result.
\end{itemize}
To sum up, in both cases, we have $\chi_{\cyc}|_{H} \neq \eta|_{H}$.

(2.b) Moreover, from (2.a) we see that the character $(\chi_{\cyc}\eta)|_{H}$ has finite image with cardinality $p-1$, which is prime to $p$ since $p$ is odd. This forces $(\chi_{\cyc}\eta)|_{\Gal_{L}}$ to be nontrivial since $L/\QQ_{p,\cyc}$ is a pro-$p$ extension.
\end{proof}

Combining Proposition \ref{prop:nonsplit-mult} and Lemma \ref{lem:nontrivial_criterion}, we see that in the special case where $F_{v} = \QQ_p$ and $L/\QQ_{p,\cyc}$ is a Galois extension of $p$-power degree, $E(L)[p^{\infty}]$ is finite. This enables us to apply Theorems \ref{thm:verticalcontrol} and \ref{thm:maintheorem} to obtain the following theorem.

\begin{theorem}[Main theorems for elliptic curves: nonsplit multiplicative reduction case] \label{thm:appecmultcyc_nonsplit}
    Let $E$ be an elliptic curve over $F$ with potential multiplicative reduction at all places $v$ of $F$ above $p$. Assuming \dref{item:spl}, and if $E$ has additive reduction at any place $v$ of $F$ above $p$, we assume that $p \geq 5$, then we have the following results.
\begin{enumerate}[label = \rm (\arabic*)]
\item \emph{(Vertical control theorem)} The Pontryagin dual of the restriction map
    \[
    \res_{\tilF/F_{\cyc}}: \Sel(E/F_{\cyc}) \rightarrow \Sel(E/\tilF)^{\Gal(\tilF/F_{\cyc})}
    \]
    is a pseudo-isomorphism of $\ZZ_p\lrbracket{\Gal(F_{\cyc}/F)}$-modules, with kernel and cokernel being finitely generated torsion $\ZZ_p$-modules.
    \item \emph{(Conjecture \ref{myconj:Bminus})} If $Y(E/F_{\cyc})$ is a finitely generated $\ZZ_p$-module, then $Y(E/\tilF)$ is a finitely generated $\ZZ_p\lrbracket{\Gal(\tilF/F_{\cyc})}$-module.
    \item \emph{(Conjecture \ref{myconj:Bplus})} If $Y(E/F_{\cyc})$ is a pseudonull $\ZZ_p\lrbracket{\Gal(F_{\cyc}/F)}$-module, then $Y(E/\tilF)$ is a pseudonull  $\ZZ_p\lrbracket{\Gal(\tilF/F)}$-module.
\end{enumerate}
\end{theorem}

\subsubsection{General cases}
We are also ready to state the following result for general elliptic curves $E$ over $F$. Let $S_{p,F}$ be the set of all places of $F$ above $p$, and let $S_{p, F}^{\mathrm{p.good}}$ (resp. $S_{p, F}^{\mathrm{mult}}$, $S_{p, F}^{\mathrm{p. mult}}$) denote its subset consisting of those at which $E$ has potentially good reduction (resp. multiplicative reduction, additive potential multiplicative reduction). We focus on the case $F_{\infty} = F_{\cyc}$.

\begin{theorem}[Main theorems for elliptic curves over $F$] \label{thm:appecoverF}
    Let $E$ be an elliptic curve over $F$. Suppose that 
    \begin{condition}
    \item[(Spl${}^{-}$)] for all $v \in S_{p}^{\mathrm{mult}} \cup S_{p}^{\mathrm{p.mult}}$, we have $F_{v} = \QQ_p$.  \label{item:splitminus}
    \item[(Large)] suppose that $p \geq 5$ when $S_{p}^{\mathrm{p.mult}} \neq \emptyset$.\label{item:large}
    \end{condition}
    Then we have the following results.
\begin{enumerate}[label = \rm (\arabic*)]
\item \emph{(Vertical control theorem)} The Pontryagin dual of the restriction map
    \[
    \res_{\tilF/F_{\cyc}}: \Sel(E/F_{\cyc}) \rightarrow \Sel(E/\tilF)^{\Gal(\tilF/F_{\cyc})}
    \]
    is a pseudo-isomorphism of $\ZZ_p\lrbracket{\Gal(F_{\cyc}/F)}$-modules, with kernel and cokernel being finitely generated torsion $\ZZ_p$-modules.
    \item \emph{(Conjecture \ref{myconj:Bminus})} If $Y(E/F_{\cyc})$ is a finitely generated $\ZZ_p$-module, then $Y(E/\tilF)$ is a finitely generated $\ZZ_p\lrbracket{\Gal(\tilF/F_{\cyc})}$-module.
    \item \emph{(Conjecture \ref{myconj:Bplus})} If $Y(E/F_{\cyc})$ is a pseudonull $\ZZ_p\lrbracket{\Gal(F_{\cyc}/F)}$-module, then $Y(E/\tilF)$ is a pseudonull  $\ZZ_p\lrbracket{\Gal(\tilF/F)}$-module.
\end{enumerate}
\end{theorem}
\begin{proof}
    At the places $v$ over $p$ where $E$ has potentially good reduction, Greenberg's finiteness condition follows from Proposition \ref{prop:imai}. At the places $v$ over $p$ where $E$ has multiplicative reduction, it follows either from a combination of Proposition \ref{prop:imaimult}, Lemma \ref{lem:q_E_divisible} and the discussion in Section \ref{sec:exampleclassgroup} concerning the finiteness of $\mu_{p^{\infty}}(L)$ in the split case, or from a combination of Proposition \ref{prop:nonsplit-mult} and Lemma \ref{lem:nontrivial_criterion} in the nonsplit case, both under condition \dref{item:splitminus}. The theorem is then a corollary of Theorem \ref{thm:maintheorem}.
\end{proof}

\subsection{Example: Classical cuspidal eigenforms} \label{sec:crystalline}
More generally, we consider a cuspidal newform $f \in S_{k}(\Gamma_{0}(N), \epsilon_{f})$ of \emph{even} weight $k \geq 2$, level $N$, and nebentypus character $\epsilon_{f}$, with Fourier expansion
\[
f = \sum_{n=1}^{\infty} a_n(f)q^n \in F\lrbracket{q},
\]
where $F := \QQ(f)$ is the coefficient field of $f$, i.e. the number field generated by all Fourier coefficients of $f$. Let $L$ be a finite extension of the completion of $F$ at a chosen prime above $p$ (recall that $p \neq 2)$, and let $\calO$ be its ring of integers.

There is an absolutely irreducible $\Gal_{\QQ}$-representation associated with $f$ and $L$:
\[
\rho_{f}: \Gal_{\QQ} \rightarrow \Aut_{L}(V_{f})
\]
where $V_{f}$ is a two-dimensional $L$-vector space, such that $\rho_{f}$ is unramified at all primes $\ell \nmid pN$ and for any such $\ell$, the characteristic polynomial of $\rho_{f}(\Frob_{\ell})$ is
\[
P_{\ell}(X) := X^{2} - a_{\ell}(f) X + \epsilon_{f}(\ell) \ell^{k-1}.
\]
We take $T_{f}$ to be any fixed $\Gal_{\QQ}$-stable $\calO$-lattice in $V_{f}$ and let $A_{f} := V_{f}/T_{f}$. Let $S$ be a finite set of places containing the infinite places and the finite places of $F$ lying over prime divisors of $Np$. For any algebraic extension $\calL/F$ inside $F_{S}$, we shall consider the fine Selmer group $\Sel_{S}(A_{f}/\calL)$ and its Pontryagin dual, and write the former as $\Sel_{S}(f/\calL)$ for simplicity.

The main tool for the local finiteness result is the following theorem of Y. Ozeki, proved as \cite[Theorem 1.2]{ozeki2020}. Here we use a simplified version (taking $k = \QQ_p$, $\pi = p$, and $K$ to be $K_{\lambda}$ for any place $\lambda$ of $K$ lying over $p$ in \textit{loc. cit.}).

\begin{theorem} \label{thm:ozeki}
Let $X$ be a proper smooth variety over $K_{\lambda}$ with potentially good reduction, where $\lambda$ is a finite place of $K$ lying over $p$. Let $V$ be a $\Gal_{K_{\lambda}}$-stable subquotient of $\rmH^{i}_{\et}(X_{\barK_{\lambda}}, \QQ_p)$ for $i > 0$. If $L$ is a finite extension of $K_{\lambda}(\mu_{p^{\infty}})$ and $\rmH^{0}(L, V) \neq 0$, then $i$ is even.
\end{theorem}

We then have the following proposition on the finiteness condition for modular forms.

\begin{proposition} \label{prop:verifylfmodularform}
Let $F_{\infty}/F$ be the cyclotomic $\ZZ_p$-extension of $F$. If $p$ does not divide $N$ and $k \geq 2$ is even, then $\rmH^{0}(L, V_{f}) = 0$, or equivalently $\rmH^{0}(L, A_{f})$ is finite, for any finite extension $L/F_{\infty, v}$, where $v$ is a place of $F_{\infty}$ lying over $p$.
\end{proposition}

\begin{proof}
It follows from the construction of $V_{f}$ by Deligne that $V_{f}$ appears as a subquotient of $\rmH^{k-1}_{\et}(\calW_{k-2, \barQQ}, \QQ_p)$, where $\calW_{k-2}$ is the \emph{Kuga-Sato variety} of weight $k-2$. It follows from \cite[Lemme 5.4]{deligne1971} that $\calW_{k-2}$ is a proper smooth variety over $K_{\lambda}$. It follows from \cite[Appendix of Brian Conrad]{bertolini2013} that $\calW_{k-2}$ has good reduction over $K_{\lambda}$. The proposition then follows from Theorem \ref{thm:ozeki} by noting that $k-1$ is an odd number and $p$ does not divide the level $N$ of $f$.
\end{proof}

Therefore, condition \dref{eq:lf} is verified. We are now ready to apply Theorems \ref{thm:verticalcontrol} and \ref{thm:maintheorem} to the arithmetic of cuspidal newforms. The following result combines them with Proposition \ref{prop:verifylfmodularform}.

\begin{theorem}[Main theorems for classical cuspidal newforms] \label{thm:appmodforms}
    Let $f \in S_{k}(\Gamma_{0}(N))$ be a cuspidal newform of weight $k \geq 2$ even and $p \nmid N$. Then we have the following results.
\begin{enumerate}[label = \rm (\arabic*)]
    \item \emph{(Vertical control theorem)} The Pontryagin dual of the restriction map
    \[
    \res_{\tilF/F_{\cyc}}: \Sel(f/F_{\cyc}) \rightarrow \Sel(f/\tilF)^{\Gal(\tilF/F_{\cyc})}
    \]
    is a pseudo-isomorphism of $\calO\lrbracket{\Gal(F_{\cyc}/F)}$-modules, with kernel and cokernel being finitely generated torsion $\calO$-modules.
    \item \emph{(Conjecture \ref{myconj:Bminus})} If $Y(f/F_{\cyc})$ is a finitely generated $\calO$-module, then $Y(f/\tilF)$ is a finitely generated $\calO\lrbracket{\Gal(\tilF/F_{\cyc})}$-module.
    \item \emph{(Conjecture \ref{myconj:Bplus})} If $Y(f/F_{\cyc})$ is a pseudonull $\calO\lrbracket{\Gal(F_{\cyc}/F)}$-module, then $Y(f/\tilF)$ is a pseudonull  $\calO\lrbracket{\Gal(\tilF/F)}$-module.
\end{enumerate}
\end{theorem}

\subsection{Example: Hida families of $p$-ordinary cuspforms} In this section, we apply Theorem \ref{thm:maintheorem} to Hida families of $p$-ordinary cuspforms, for which Theorem \ref{thm:bigimai} plays an important role in verifying Greenberg's finiteness condition. We recall the basic setup of classical Hida theory following \cite[Section 2.6]{skinner2016}; the primary sources are Hida's works \cite{hida1986galois,hida1988measure}.

Retaining the notation of Section \ref{sec:crystalline}, let $f \in S_{k}(\Gamma_{0}(N))$ be a cuspidal newform that is $p$-ordinary; that is, $a_p(f) \in \calO^{\times}$. We write $N = p^r M$ with $p \nmid M$. It follows from \cite[Lemma 2.1.2]{skinner2016} that $r = 0$ or $1$. Let $\Lambda_{\ord} := \calO\lrbracket{X}$. Hida proved that there is a finite local $\Lambda_{\ord}$-domain $\II$ and a formal $q$-expansion
\[
\bff = \sum_{n=1}^{\infty} a_n q^{n} \in \II \lrbracket{q}, \, a_1 = 1,
\]
satisfying:
\begin{enumerate}[label = (H-\arabic*)]
    \item $\II = \Lambda_{\ord}[\{a_{\ell}: \ell \text{ is a prime number} \}]$;
    \item if $\phi: \II \rightarrow \barQQ_p$ is a continuous $\calO$-algebra homomorphism such that $\phi(1+X) = (1+p)^{k^{\prime}}$ with $k^{\prime} \equiv k \pmod{p-1}$, then $\sum_{n=1}^{\infty} \phi(a_n) q^n$ is the $q$-expansion of a \emph{$p$-stabilized} newform, in the sense that there is a newform $f_{\phi} \in S_{k^{\prime}}(\Gamma_0(M))$ and an embedding $\QQ(f_{\phi}) \hookrightarrow \barQQ_p$ such that $\phi(a_{\ell}) = a_{\ell}(f_{\phi})$ for all primes $\ell \neq p$, and $\phi(a_p)$ is the unit root of the polynomial $x^{2} - a_p(f_{\phi})x + p^{k^{\prime} - 1}$; \label{item:H2}
    \item there is a continuous $\calO$-algebra homomorphism $\phi_0: \II \rightarrow \calO$ such that $\phi_0(1+X) = (1+p)^{k}$ and $\phi_0(a_{\ell}) = a_{\ell}(f)$ for $\ell \neq p$ and $\phi_0(a_p)$ is the unit root of $x^2 - a_p(f)x+p^{k-1}$ if $r = 0$ and $\phi_0(a_p) = a_p(f)$ if $r = 1$.
\end{enumerate}
We call $\bff$ an \emph{$\II$-adic cuspidal Hida family} passing through $f$, and call the $\calO$-algebra homomorphism $\phi: \II \rightarrow \barQQ_p$ in \ref{item:H2} an \emph{arithmetic specialization} of $\bff$ of \emph{weight} $k$.

By enlarging $\calO$ when necessary, we may assume that $\calO$ is integrally closed in $\II$. In the remaining part of this section, we put the following assumption.
\begin{condition}
    \item[(irred)] The residual Galois representation $\barrho_f$ is irreducible. \label{eq:irred}
\end{condition}
Then there is a free $\II$-module $\TT$ of rank two and a continuous Galois representation
\[
\rho_{\bff}: \Gal_{\QQ} \rightarrow \Aut_{\II}(\TT)
\]
which is unramified at each $\ell \nmid pN$ and such that, for any such prime, $\tr(\rho_{\II}(\Frob(\ell))) = a_{\ell} \in \II$. In particular, when $\phi: \II \rightarrow \calO$ is $\phi_0$ or an arithmetic specialization of weight $k^{\prime}$, we see that
\[
T_{f_{\phi}} = \TT \otimes_{\II, \phi} \calO
\]
is a $\Gal_{\QQ}$-stable $\calO$-lattice in $\TT \otimes_{\II, \phi} L \simeq V_{f_{\phi}}$.

We shall also use the following assumption.
\begin{condition}
    \item[(even)] The weight of the newform $f$ is even. As a result, any arithmetic specialization of $\bff$ has even weight since its weight $k^{\prime}$ satisfies $k \equiv k^{\prime} \pmod{p-1}$ and $p$ is an odd prime number. \label{eq:even}
\end{condition}

Let $\AA := \TT \otimes_{\II} \II^{\vee}$. Then we have the following result.

\begin{proposition} \label{prop:hidabigimai}
With the notation above, suppose that conditions \dref{eq:irred} and \dref{eq:even} hold. Then $\AA(L)$ is a cotorsion $\II$-module for any finite extension $L/F_{\cyc,v}$, where $v$ is a place of $F_{\cyc}$ lying over $p$.
\end{proposition}
\begin{proof}
    Recall that $\II$ is an integral domain. It satisfies Serre's condition $(S_1)$, and $\Spec \II$ is irreducible. At any arithmetic specialization $\phi$ of $\bff$ of weight $k^{\prime} > 2$, by \ref{item:H2}, we obtain a newform $f_{\phi} \in S_{k^{\prime}}(\Gamma_0(M))$, where $M$ is the prime-to-$p$ part of $N$. \footnote{For the precise reason why $f_{\phi}$ has level prime to $p$, \cite[Lemma 2.1.2]{skinner2016} proves that, for any $p$-ordinary cuspidal newform $g \in S_{k}(\Gamma_1(Q))$, if $p \mid Q$, then $p \mid\mid Q$, $k = 2$, and $a_p(f) = \pm 1$.} Therefore, for any arithmetic specialization $\phi$ of $\bff$ of weight $k^{\prime} > 2$, the newform $f_{\phi}$ has level prime to $p$. Under assumption \dref{eq:even}, $k^{\prime}$ is even. It follows from Proposition \ref{prop:verifylfmodularform} that $A_{f_{\phi}}$ satisfies Greenberg's finiteness condition \dref{eq:lf}. We take $\wp := \ker(\phi)$. Then $\AA^{(\wp)} = A_{f_{\phi}}$, and the proposition follows from Theorem \ref{thm:bigimai}.
\end{proof}

Let $S$ be a finite set of places containing infinite places and finite places of $F$ lying over prime divisors of $Np$. We write the fine Selmer group $\Sel_{S}(\AA/\calL)$ as $\Sel_{S}(\bff/\calL)$ for simplicity, where $\calL/F$ is any algebraic extension inside $F_{S}$.

The following result combines Theorems \ref{thm:verticalcontrol} and \ref{thm:maintheorem} with Proposition \ref{prop:hidabigimai}.

\begin{theorem}[Main theorems for Hida families of $p$-ordinary cuspforms] \label{thm:apphida}
    With the notation above, assume that conditions \dref{eq:irred} and \dref{eq:even} are satisfied. Then we have the following results.
\begin{enumerate}[label = \rm (\arabic*)]
\item \emph{(Vertical control theorem)} The Pontryagin dual of the restriction map
    \[
    \res_{\tilF/F_{\cyc}}: \Sel(\bff/F_{\cyc}) \rightarrow \Sel(\bff/\tilF)^{\Gal(\tilF/F_{\cyc})}
    \]
    is a pseudo-isomorphism of $\II\lrbracket{\Gal(F_{\cyc}/F)}$-modules, with kernel and cokernel being finitely generated torsion $\II$-modules.
\item \emph{(Conjecture \ref{myconj:Bminus})} If $Y(\bff/F_{\cyc})$ is a finitely generated $\II$-module, then $Y(\bff/\tilF)$ is a finitely generated $\II\lrbracket{\Gal(\tilF/F_{\cyc})}$-module.
    \item \emph{(Conjecture \ref{myconj:Bplus})} If $\II$ is a Cohen--Macaulay ring and $Y(\bff/F_{\cyc})$ is a pseudonull $\II\lrbracket{\Gal(F_{\cyc}/F)}$-module, then $Y(\bff/\tilF)$ is a pseudonull  $\II\lrbracket{\Gal(\tilF/F)}$-module.
\end{enumerate}
\end{theorem}

\begin{remark} \label{rem:branchCM}
    Suppose $\II$ is a branch of the local component $\widetilde{\II}$ of Hida's ordinary Hecke algebra $\HH_{N}^{\ord}$ of tame level $N$. If the residual Galois representation $\barrho_{\bff}$ of the Hida family $\bff$ is irreducible, then $\widetilde{\II}$ is a Gorenstein ring (see \cite[Theorem 7.10, Remark 7.11]{ochiai2025iwasawa3} for concrete results and explanations) and hence Cohen--Macaulay. However, this does not imply that the branch $\II$ is a Gorenstein ring. Moreover, if $\II$ is Cohen--Macaulay, then $\widetilde{\II}$ need not be Cohen--Macaulay, and vice versa. Explicit examples can be found in \cite[Appendix A]{ochiai2025iwasawa3}. Nevertheless, the condition that the branch $\II$ is Cohen--Macaulay (in fact, isomorphic to a one-variable power series ring $\calO\lrbracket{X}$) holds in a fairly general setting; see \cite[Remark 8.11 (3)]{ochiai2025iwasawa3}.
\end{remark}

In future work, we hope to consider broader $p$-adic families of Galois representations.

\appendix

\section{On the coinvariant criterion for pseudonullity}
\label{sec:appA}
In this appendix, we make further remarks on the coinvariant criterion for pseudonullity (i.e. Theorem \ref{catenary}).

\subsection{Geometric interpretation} \label{sec:rmkgeom}
The geometric reason for the Cohen-Macaulay hypothesis in Theorem \ref{catenary} is as follows. Since
\[
\Supp(M/TM) = V(T) \cap \Supp(M),
\]
the hypothesis that $M/TM$ is pseudonull over $\Lambda/T\Lambda$ means that every irreducible component of $V(T) \cap \Supp(M)$ has codimension at least $2$ inside $V(T)$. Because $T$ is a regular element, $V(T)$ is an effective Cartier divisor, hence has codimension $1$ in $\Spec(\Lambda)$. Therefore $V(T) \cap \Supp(M)$ has codimension at least 3 in $\Spec(\Lambda)$.

To prove that $M$ is pseudonull, it suffices to show that $\Supp(M)$ has no irreducible component of codimension at most $1$. Suppose instead that $Y$ is an irreducible component of $\Supp(M)$ of codimension $1$. Since $T$ lies in the Jacobson radical of $\Lambda$, every closed point of $\Spec(\Lambda)$ lies in $V(T)$. As every irreducible closed subset of a Noetherian scheme contains a closed point, we have $Y \cap V(T) \neq \emptyset$.

If $\Lambda$ is Cohen--Macaulay, then intersecting a codimension-one irreducible closed subset $Y$ with the effective Cartier divisor $V(T)$ increases its codimension \emph{by exactly one} once the intersection is nonempty (while in general the intersection may further increase the codimension). Thus every irreducible component of $Y \cap V(T)$ has codimension $2$ in $\Spec(\Lambda)$, contradicting the fact that $V(T) \cap \Supp(M)$ has codimension at least $3$ in $\Spec(\Lambda)$. Therefore $\Supp(M)$ has no codimension-one components, and $M$ is pseudonull.

\subsection{The Cohen--Macaulay assumption}
In Theorem \ref{catenary}, the only place where we used the Cohen--Macaulay property of the coefficient ring $\Lambda$ is \eqref{eq:CM_condition}. A weaker sufficient condition is that $\Lambda$ be Noetherian, catenary, and locally equidimensional.
\begin{definition}
    Let $\Lambda$ be a commutative Noetherian ring. We say that $\Lambda$ is
    \begin{enumerate}
        \item \emph{catenary} if, for any pair of prime ideals $\frp \subseteq \frq$, there exists an integer bounding the lengths of all finite chains of prime ideals $\frp = \frp_0 \subset \frp_{1} \subset \cdots \subseteq \frp_{e} = \frq$, and all maximal such chains have the same length (see \cite[\href{https://stacks.math.columbia.edu/tag/00NI}{Tag 00NI}]{stacks-project});
        \item \emph{equidimensional} if all irreducible components of $\Spec(\Lambda)$ have the same dimension; and
        \item \emph{locally equidimensional} if, for every prime ideal $\frq$ of $\Lambda$, all irreducible components of $\Spec(\Lambda_{\frq})$ have the same dimension. Equivalently, for every minimal prime ideal $\frr$ of $\Lambda_{\frq}$, $\dim(\Lambda/\frr) = \dim(\Lambda_{\frq})$.
    \end{enumerate}
\end{definition}

We give the following improvement of Theorem \ref{catenary} under these weaker assumptions.

\begin{theorem}[Improved coinvariant criterion for pseudonullity]\label{catenary2}
    Let $\Lambda$ be a commutative Noetherian ring and assume that $\Lambda$ is catenary and locally equidimensional. Let $T$ be a regular element in $\Lambda$ and assume $T$ is in the Jacobson radical. Let $M$ be a finitely generated $\Lambda$-module. If $M/TM$ is a pseudonull $\Lambda/T\Lambda$-module, then $M$ is a pseudonull $\Lambda$-module.
\end{theorem}

We remark that a catenary, locally equidimensional Noetherian ring is not necessarily Cohen--Macaulay, so Theorem \ref{catenary2} strictly improves Theorem \ref{catenary}.

\begin{proof}
    The proof of Theorem \ref{catenary} carries over once we prove the displayed height equality \eqref{eq:CM_condition} under the assumption that $\Lambda$ is catenary and locally equidimensional. This is verified in Proposition \ref{prop:height-additivity}.
\end{proof}

\begin{proposition}\label{prop:height-additivity}
Let $A$ be a commutative Noetherian ring. Assume that $A$ is
catenary and locally equidimensional. Then, for every pair of prime
ideals $\frp\subseteq\frq$ of $A$, one has
\[
\Ht_{A}(\frq) = \Ht_{A}(\frp) + \Ht_{A/\frp} (\frq/\frp).
\]
\end{proposition}

\begin{proof}
Localize $A$ at $\frq$, and write
\[
B=A_{\frq},\quad \frm=\frq A_{\frq},\quad \frp=\frp A_{\frq}.
\]
Since localization preserves catenarity (\cite[\href{https://stacks.math.columbia.edu/tag/0AUN}{Tag 0AUN}]{stacks-project}), $B$ is catenary, and $B/{\frr}$ is also catenary because quotients preserve catenarity (\cite[\href{https://stacks.math.columbia.edu/tag/00NK}{Tag 00NK}]{stacks-project}). By the local equidimensionality assumption, $B$ is equidimensional.

Put $d = \dim B$. Let $\frr\subseteq\frp$ be a minimal prime of $B$. Since $B$ is equidimensional, $\dim(B/\frr)=d$, or equivalently, $\Ht_{B/\frr} (\frm/\frr)=d$. Catenarity of the \emph{Noetherian domain} $B/\frr$ applied to the chain $\frr\subseteq\frp\subseteq\frm$ gives (see \cite[(14.B)]{MR575344}):
\[
d = \Ht_{B/\frr} (\frm/\frr) = \Ht_{B/\frr}(\frp/\frr) + \Ht_{B/\frp} (\frm/\frp).
\]
In particular,
\[
\Ht_{B/\frr}
   (\frp/\frr)
=
d-\dim(B/\frp).
\]
The right-hand side is independent of the choice of the minimal prime
$\frr\subseteq\frp$.

Every chain of prime ideals ending at $\frp$ can be extended downward to a minimal prime contained in $\frp$. It follows that
\[
\Ht_{B}(\frp) = \max_{\substack{\frr\in\operatorname{MinSpec}(B), \, \frr\subseteq\frp}} \Ht_{B/\frr} (\frp/\frr).
\]
Since all the terms on the right are equal by the preceding paragraph, we obtain
\[
\Ht_{B}(\frp) = d-\dim(B/\frp), 
\]
or equivalently,
\[
\dim B = \Ht_{B}(\frp) + \dim(B/\frp).
\]
We expand these three summands. The following identities follow from the definitions.
\begin{itemize}
    \item $\dim B=\dim A_{\frq} = \Ht_{A}(\frq)$,
    \item $\Ht_{B}(\frp) = \dim B_{\frp} = \dim A_{\frp} = \Ht_{A}(\frp)$, and 
    \item $\dim(B/\frp) = \dim\left((A/\frp)_{\frq/\frp}\right) = \Ht_{A/\frp}(\frq/\frp)$.
\end{itemize}
Substitution gives the required result.
\end{proof}

In applications of the main theorem of this paper, we hope to take the coefficient ring $R$ to be a Galois deformation ring. However, not all Galois deformation rings are Cohen--Macaulay. We therefore hope to relax the conditions in Theorems \ref{catenary} and \ref{catenary2}. In what follows, we present counterexamples involving the conditions in Theorems \ref{catenary} and \ref{catenary2}.

\subsubsection*{Counterexample: catenary condition is indispensable in Theorem \ref{catenary2}} \label{sec:counterexample}
Notice that a local domain is equidimensional and locally equidimensional.
In the Appendix of \cite{ogoma1980non}, Ogoma constructs a Noetherian local normal domain $T$ with $\dim T = 3$ and a height 1 prime ideal $ \frp$ such that $\dim (T/\frp)=1$. Hence
\[
\Ht_{T}(\frp) + \dim (T/\frp) < \dim T = 3, 
\]
which implies that $T$ is \emph{not} catenary.

Let $\frm$ be the maximal ideal of $T$. Take any element $r$ in $\frm \smallsetminus \frp$. Then $r$ is a regular element since $T$ is a domain, and $r$ is in the Jacobson radical of $T$ since $T$ is local. Since $r \not\in \frp$, any prime ideal of $T/\frp$ containing $ (r)+\frp/\frp$ is a nonzero prime. Since $T/\frp$ is one-dimensional and local with maximal ideal $\frm/\frp$, the only such prime containing $ (r)+\frp/\frp$ is its maximal ideal $\frm/\frp$. Hence, we have $V(\frp)\cap V(r)=\{\frm\} $. 

If we take $M=T/\frp$, following the manner of the geometric explanation in Section \ref{sec:rmkgeom}, we consider $\Supp(M) = V(\frp)$, which is an irreducible closed subset of $\Spec(T)$. Firstly, $M$ is not a pseudonull $T$-module: after localizing at the height-one prime $\frp$, we have $M_{\frp} = T_{\frp}/\frp T_{\frp} \neq 0$. However, $M/rM$ is a pseudonull $T/rT$ module. Indeed, $M/rM = T/(\frp + (r))T$, and hence 
\[
\Supp_{T}(M/rM) = V(\frp + (r)) = V(\frp) \cap V(r) = \{\frm\},
\]
and hence $\Supp_{T/rT}(M/rM) = \{\frm/(r)\}$. Since $T/rT$ is a local ring with maximal ideal $\frm/(r)$, the height of $\frm/(r)$ in $T/rT$ is at least $2$. Indeed, since $T$ is a local domain of dimension $3$,  \cite[Corollary 11.18]{atiyah1969introduction} tells us that $\dim(T/rT) = 2$ since $r$ is a regular element in $T$. Since $T/rT$ is a local ring, the dimension of the ring is, by definition, the height of its maximal ideal $\frm/(r)$. Therefore, the support of $M/rM$ in $\Spec(T/rT)$ is a singleton consisting of an ideal of $T/rT$ of height $2$, proving that $M/rM$ is indeed a pseudonull $T/rT$-module.

\subsubsection*{Counterexample: local equidimensionality is indispensable in Proposition \ref{prop:height-additivity}, as well as in Theorem \ref{catenary2}}

Catenarity says that saturated chains between a fixed pair of primes have a common length. Without an equidimensionality or biequidimensionality condition, it does not imply the absolute height additivity displayed in \eqref{eq:CM_condition}.

Let $k$ be a field, and consider $S=k \lrbracket{a,b,c,d}$ and its prime ideals $P = (a)$ and $Q = (b,c,d)$. Consider
\[
  R=S/(ab,ac,ad) = S/(P \cap Q),\quad \frp=(b,c,d)R, 
\]
and put $T=a+b$ and $M=R/\frp$. We divide the counterexample into the following steps.

(1) Since $P$ and $Q$ are incomparable prime ideals, the minimal primes of $R$ are exactly $\frp = (b,c,d) R$ and $\frq = (a)R$. The ring $R$ is therefore reduced. The ring $R$ is a complete Noetherian local ring with maximal ideal $\frm = (a,b,c,d)R$, and hence catenary.

(2) The two irreducible components of $\Spec(R)$ have different dimensions:
\[
R/\frq \simeq S/(a) \simeq k \lrbracket{b,c,d}, \, \text{so} \dim(R/\frq) = 3,
\]
while 
\[
R/\frp \simeq S/(b,c,d) \simeq k \lrbracket{a}, \, \text{so} \dim(R/\frp) = 1.
\]
It follows that $\dim R = 3$ but the components $R/\frp$ and $R/\frq$ have different dimensions. Thus $R$ is not equidimensional. Since $R = R_{\frm}$, this also shows that $R$ is not locally equidimensional.

(3) On the component $R/\frp$, there is a saturated chain $\frp \subsetneq \frm$ of ideals of $R$ with starting point $\frp$ of length $1$. On the other component $R/\frq$, there is a saturated chain 
\[
\frq \subsetneq (a,b)R \subsetneq (a,b,c)R \subsetneq \frm
\]
of ideals of $R$ with starting point $\frq$ of length $3$. This gives the precise failure of the absolute height additivity in Proposition \ref{prop:height-additivity}: 
\[ 
\Ht_{R}(\frm) = 3, \quad \Ht_{R}(\frp) = 0, \quad \Ht_{R/\frp}(\frm/\frp) = 1.
\]
Hence $\Ht_{R}(\frm) \neq \Ht_{R}(\frp) + \Ht_{R/\frp}(\frm/\frp)$. This is the central mechanism of the counterexample.

(4) Since $R$ is reduced and Noetherian, its zero divisors are precisely the union of its minimal primes (combining \cite[\href{https://stacks.math.columbia.edu/tag/00LD}{Tag 00LD}]{stacks-project} and \cite[\href{https://stacks.math.columbia.edu/tag/031R}{Tag 031R}]{stacks-project}). Since $a+b$ belongs to neither minimal prime, $T$ is a non-zero-divisor, and hence $T$ is regular.

(5) The prime $\frp$ is minimal, so $\Ht_{R}(\frp) = 0$. Since $\frp$ is minimal, the localization $M_\frp$ remains nonzero and therefore $M$ is not pseudonull over $R$. On the other hand, 
\[
R/TR = \dfrac{k \lrbracket{a,b,c,d}}{(ab,ac,ad, a+b)} \simeq \dfrac{k\lrbracket{a,c,d}}{(a^2,ac,ad)}
\]
by using $b = -a$. Since $\sqrt{(a^2,ac,ad)}=(a)$, and Krull dimension is unchanged upon passing to the radical, it follows that
\[
\dim(R/TR)=\dim\dfrac{k \lrbracket{a,c,d}}{(a^2,ac,ad)} = \dim\dfrac{k \lrbracket{a,c,d}}{(a)} = \dim k \lrbracket{c,d} = 2.
\]
The local ring $R/TR$ has dimension $2$, and the support of $M/TM \simeq k$ is its maximal ideal, of height $2$. Thus $M/TM$ is pseudonull over $R/TR$.

\section{On elliptic curves with potential multiplicative reduction}

In this appendix, we add details from Section \ref{sec:app_potential_multiplicative}, concerning arguments that are already well known to experts. We include them here to avoid interrupting the flow of the article.

\subsection{On the Tate uniformization} \label{app:tate}

Let $E$ be an elliptic curve over a number field $F$ and let $v$ be a place of $F$ above $p$ such that $E$ has split multiplicative reduction at $v$. In this section, we describe the map $E[p^{\infty}] \rightarrow Q$ in \eqref{eq:multiplicative} more explicitly. This description will be used in the proof of Proposition \ref{prop:kerneldelta}.

Throughout this section, for the Tate period $q_{E} \in \barF_{v}^{\times}$, we fix a compatible system of $p^{n}$-th roots of the Tate period $q_{E}$, namely 
\[
q_{E}^{1/p^{\infty}} := (q_{E}^{1/p}, q_{E}^{1/p^2}, \ldots, q_{E}^{1/p^{n}}, \ldots) \in (\barF_{v}^{\times})^{\ZZ}
\]
such that $(q_{E}^{1/p^{n+1}})^{p} = q_{E}^{1/p^{n}}$.

Under Tate's uniformization, let $P_{u} \in E[p^{\infty}]$ be represented by a class $[u] \in \barF_{v}^{\times}/q_{E}^{\ZZ}$ with $u \in \barF_{v}^{\times}$. Then $P_{u}$ has order dividing $p^{n}$ if and only if $u^{p^n} \in q_{E}^{\ZZ}$. Therefore, there exists an integer $m \in \ZZ$ such that $u^{p^{n}} = q_{E}^{m}$. Using our fixed compatible system $q_{E}^{1/p^{\infty}}$, we may write $u = \zeta q_{E}^{m/p^{n}}$ for some $p^{n}$-th root of unity $\zeta$. The factor $\zeta$ belongs to $\mu_{p^{\infty}}$ and hence disappears after passing to the quotient $Q := E[p^{\infty}]/\mu_{p^{\infty}}$. Consequently, the image of $P_{u}$ in $Q$ is actually the class $[q_{E}^{m/p^{n}}]$. Notice that this class only depends on
\[
\frac{m}{p^n}\in \frac{1}{p^n}\ZZ/\ZZ\subset\QQ_p/\ZZ_p,
\]
because replacing $m$ by $m+p^n$ changes $q_E^{m/p^n}$ by the factor
$q_E$, which is trivial in the Tate quotient $\barF_v^{\times} / q_E^\ZZ$. We therefore obtain a well-defined identification
\[
\iota_E: Q \xrightarrow{\sim} \QQ_p/\ZZ_p,
\]
which sends the class of the point represented by $q_E^{m/p^n}$ to
$m/p^n$.

Indeed, it suffices to verify that this description is independent of the choice of $n$. Suppose that the same point $P_u$ is also represented at level $n^{\prime}>n$, so that $u=\zeta^{\prime} q_E^{m^{\prime}/p^{n^{\prime}}}$ for some $\zeta^{\prime} \in\mu_{p^{n^{\prime}}}$. Then $q_E^{m/p^n-m^{\prime}/p^{n^{\prime}}} = \zeta^{\prime}\zeta^{-1} \in\mu_{p^\infty}$. Since $q_{E}$ has positive $v$-adic valuation, this forces $m/p^n = m^{\prime}/p^{n^{\prime}}$. Therefore, after passing to $Q$, $[q_{E}^{m/p^{n}}]$ and $[q_{E}^{m^{\prime}/p^{n^{\prime}}}]$ define the same element of $Q$ and determine the same element $m/p^{n} = m^{\prime}/p^{n^{\prime}}$ of $\QQ_p/\ZZ_p$.

Finally, the action of $\Gal_{F_v}$ on $Q$ is trivial. Indeed, for
$\sigma\in\Gal_{F_v}$, $\sigma(q_E^{m/p^n})$ is another $p^n$-th root of $q_E^m$, and hence differs from $q_E^{m/p^n}$ by an element of $\mu_{p^n}$. Therefore the Galois action becomes trivial after quotienting by $\mu_{p^\infty}$.

\subsection{On nonsplit multiplicative reduction}

In this section, we prove the following propositions used in Section \ref{sec:nonsplit}.

The following result may be well known to experts, but we have been unable to
find a reference in the literature. We therefore include a proof.

\begin{theorem}\label{thm:twist-elliptic}
Let \(K\) be a nonarchimedean local field with finite residue field \(k\) of
odd characteristic, and let \(E/K\) be an elliptic curve with potentially
multiplicative reduction. Then there exists a quadratic character
\[
    \psi\colon G_K \longrightarrow \{\pm 1\}
\]
such that the quadratic twist \(E^\psi\) has split multiplicative reduction.
More precisely:
\[
\begin{array}{c|c}
\text{reduction type of \(E/K\)} & \text{character \(\psi\)} \\ \hline
\text{split multiplicative} & \text{trivial} \\
\text{nonsplit multiplicative} & \text{nontrivial unramified quadratic} \\
\text{additive potentially multiplicative} & \text{ramified quadratic}.
\end{array}
\]
\end{theorem}

We first recall a form of the Tate curve criterion
\cite[Theorem V.5.3]{silverman1994advanced}.

\begin{theorem}[Tate curve criterion]\label{thm:tate-curve}
Let \(v\) be the valuation on \(K\), and let \(E/K\) be an elliptic curve whose
\(j\)-invariant satisfies \(v(j(E))<0\). Let \(c_4\) and \(c_6\) be the usual
invariants associated with a Weierstrass equation for \(E\), and define
\[
    \gamma(E)=-\frac{c_4}{c_6}\in K^\times/(K^\times)^2.
\]
Then:
\begin{enumerate}
    \item There is a unique \(q\in \overline{K}^{\times}\) with \(v(q)>0\)
    such that \(E\) is isomorphic over \(\overline{K}\) to the Tate curve
    \(E_q\).
    \item The curve \(E\) has split multiplicative reduction over \(K\) if
    and only if \(\gamma(E)=1\) in \(K^\times/(K^\times)^2\).
\end{enumerate}
\end{theorem}

\begin{remark}
The square class \(\gamma(E)\) is independent of the chosen Weierstrass
equation. Indeed, under a change of variables with scaling factor
\(\mu\in K^\times\), the invariants satisfy
\[
    \mu^4c_4'=c_4,
    \qquad
    \mu^6c_6'=c_6
\]
\cite[Chapter III, \S1]{silverman2009arithmetic}. Consequently,
\[
    -\frac{c_4'}{c_6'}
    =-\mu^2\frac{c_4}{c_6}
    =-\frac{c_4}{c_6}
    \quad\text{in }K^\times/(K^\times)^2.
\]
Thus a minimal Weierstrass equation is not required in the statement of
Theorem~\ref{thm:tate-curve}.
\end{remark}

\begin{remark}
By \cite[Lemma V.5.1 and Theorem V.5.3]{silverman1994advanced}, the inequality
\(v(j(E))<0\) holds if and only if \(E\) has potentially multiplicative
reduction.
\end{remark}

We now establish two lemmas needed for the proof of
Theorem~\ref{thm:twist-elliptic}.

\begin{lemma}\label{lem:c4-square}
If \(v(j(E))<0\), then \(c_4\) is a square in \(K^\times\).
\end{lemma}

\begin{proof}
Recall that
\[
    1728\Delta=c_4^3-c_6^2
    \qquad\text{and}\qquad
    j=\frac{c_4^3}{\Delta}.
\]
It follows that
\[
    \frac{c_6^2}{c_4^3}
    =1-1728\frac{\Delta}{c_4^3}
    =1-\frac{1728}{j}.
\]
Set \(t=1-1728/j\). Since \(v(j)<0\), we have
\(t\equiv 1\pmod{\mathfrak m_K}\), where \(\mathfrak m_K\) is the maximal
ideal of the valuation ring \(\mathcal O_K\). The polynomial
\(f(x)=x^2-t\) therefore has the simple root \(1\) modulo
\(\mathfrak m_K\), because the residue characteristic is odd. Hensel's lemma
then shows that \(t\) is a square in \(K^\times\).
Therefore \(c_4\) is a square in \(K^\times\).
\end{proof}

Lemma~\ref{lem:c4-square} shows that, in the Tate curve criterion, the square
class of \(\gamma(E)\) may be replaced by that of \(-c_6(E)\).

\begin{lemma}\label{lem:twist-invariants}
Let \(E/K\) be an elliptic curve, let \(d\in K^\times\), and let \(E^{(d)}\)
be the quadratic twist of \(E\) by \(d\). Then
\[
    c_4(E^{(d)})=d^2c_4(E), \qquad
    c_6(E^{(d)})=d^3c_6(E), \qquad
    \Delta(E^{(d)})=d^6\Delta(E).
\]
\end{lemma}

\begin{proof}
 By \cite[Chapter III, \S1]{silverman2009arithmetic}, we may write the elliptic curve in
 the form $E : y^2=x^3-27c_4x-54c_6$. The twisted elliptic curve then has the form
    \[
    E^{(d)} : dy^2=x^3-27c_4x-54c_6.
    \]
After the change of variables $X=dx$ and $Y=d^2y$, we obtain
    \[
     E^{(d)} : Y^2=X^3-27d^2c_4X-54d^3c_6.
    \]
Hence, 
    \[
    c_4(E^{(d)})=d^2c_4(E), \, c_6(E^{(d)})=d^3c_6(E).
    \]
By the formula $ 1728\Delta=c_4^3-c_6^2$, we have 
    \[
    \Delta(E^{(d)})=d^6\Delta(E).
    \]
The lemma is therefore proved.
\end{proof}

We are now ready to prove Theorem \ref{thm:twist-elliptic}.

\begin{proof}[Proof of Theorem~\ref{thm:twist-elliptic}]
Suppose first that \(E\) has split multiplicative reduction. No twist is
needed, so we take \(\psi\) to be the trivial character.

Now suppose that \(E\) has nonsplit multiplicative reduction, and choose a
minimal Weierstrass equation. Then \(v(\Delta)>0\) and \(v(c_4)=0\); see
\cite[Chapter VII, Proposition 5.1]{silverman2009arithmetic}. The identity
\(1728\Delta=c_4^3-c_6^2\) gives \(v(c_6)=0\). By
Theorem~\ref{thm:tate-curve} and Lemma~\ref{lem:c4-square}, the element
\(d=-c_6(E)\) is a nonsquare unit. Lemma~\ref{lem:twist-invariants} gives
\[
    -c_6\bigl(E^{(d)}\bigr)
    =-d^3c_6(E)
    =c_6(E)^4,
\]
which is a square in \(K^\times\). Hence \(E^{(d)}\) has split
multiplicative reduction. Since \(d=-c_6(E)\) is a nonsquare unit, the
extension \(K(\sqrt d)/K\) is the unramified quadratic extension. Therefore,
the associated quadratic character
\[
    G_K \longrightarrow \operatorname{Gal}(K(\sqrt d)/K)
    \longrightarrow \{\pm1\}
\]
is nontrivial and unramified.

Suppose that \(E\) has additive potentially multiplicative reduction. By
\cite[Chapter VII, Proposition 5.1]{silverman2009arithmetic}, we know that
\(v(\Delta)>0\) and \(v(c_4)>0\). The formula
\(1728\Delta=c_4^3-c_6^2\) then shows that \(v(c_6)>0\). Hence,
\(-c_6\) is not a unit and is a nonsquare in \(\mathcal O_K\) by
Theorem~\ref{thm:tate-curve}. Similarly, we can twist the elliptic curve
\(E\) by \(d=-c_6(E)\), and \(E^{(d)}\) has split multiplicative reduction. By \cite[Chapter VII, Proposition 5.4]{silverman2009arithmetic}, additive reduction remains additive reduction after an unramified extension.
 Hence, the associated quadratic character
\[
    \Gal_K \longrightarrow \operatorname{Gal}(K(\sqrt d)/K)
    \longrightarrow \{\pm1\}
\]
is ramified.
\end{proof}

\begin{proposition} \label{prop:twisted_tate_sequence}
    With the notation of Theorem \ref{thm:twist-elliptic}, we have the following short exact sequence of $\Gal_{K}$-modules:
    \[
0 \rightarrow \mu_{p^{\infty}} \otimes \eta \rightarrow E[p^{\infty}] \rightarrow \QQ_p/\ZZ_p \otimes \eta \rightarrow 0
    \]
    Here $\Gal_{K}$ acts on $\mu_{p^\infty}\otimes\eta$ through the character $\chi_{\cyc}\eta$, while it acts on $\QQ_p/\ZZ_p \otimes \eta$ through $\eta$.
\end{proposition}
\begin{proof}
    Write the Weierstrass equation of the elliptic curve $E$ as
    \[   
 E : y^2=x^3+a_2x^2+a_4x+a_6.\]
 It is easy to see that $[-1](x,y)=(x,-y)$.
 We know the elliptic curve $E^\eta$ is the elliptic curve $E$ twisted by $d=-c_6(E)$. We can write the Weierstrass equation of the elliptic curve $E^\eta$ as 
 \[
E^\eta=E^{(d)} : dy^2=x^3+a_2x^2+a_4x+a_6.
 \]
  Let $\phi : E^{(d)}\to E$ be the isomorphism over $\bar{K}$ given by $(x,y)\to (x,\sqrt{d}y)$. We can identify the $G_K$-module $E^\eta[p^\infty]$ with $E[p^\infty]$ through $\phi$.
  For any $\sigma\in G_K$, we see that
\[
\begin{aligned}
\sigma\circ\phi(x,y)
&=\sigma(x,\sqrt d\,y)
 =\bigl(\sigma(x),\sigma(\sqrt d)\sigma(y)\bigr)\\
&=\bigl(\sigma(x),(\eta(\sigma)\sqrt d)\,\sigma(y)\bigr)\\
&=\eta(\sigma)\cdot\bigl(\sigma(x),\sqrt d\,\sigma(y)\bigr)
 =\eta(\sigma)\cdot\phi\circ\sigma(x,y).
\end{aligned}
\]
Hence,
  \[
{}^\sigma\phi
 :=
\sigma\circ\phi\circ\sigma^{-1}=
\eta(\sigma)\phi. 
  \]
  Therefore, we have \[
  E^\eta[p^\infty]\cong E[p^\infty]\otimes \eta
  \] as a $G_K$-module. 
  Since $E^\eta$ has split multiplicative reduction, we have
  \[
      0 \rightarrow \mu_{p^{\infty}} \rightarrow E^{\eta}[p^{\infty}] \rightarrow \QQ_p/\ZZ_p \rightarrow 0.
  \]
  Twisting it by a quadratic character $\eta$, we have
  \[
      0 \rightarrow \mu_{p^{\infty}}\otimes \eta \rightarrow E[p^{\infty}] \rightarrow \QQ_p/\ZZ_p \otimes \eta \rightarrow 0,
  \]
  as desired.
\end{proof}

\printbibliography

\end{document}